\documentclass[11pt]{article}

\usepackage{amsfonts}
\usepackage{amsmath,amsthm,amscd,amssymb,mathrsfs,setspace}
\usepackage{latexsym,epsf,epsfig}
\usepackage{cite}
\usepackage{color}
\usepackage[hmargin=1in,vmargin=1in]{geometry}
\usepackage[T1]{fontenc}
\usepackage[utf8]{inputenc}
\usepackage{combelow}
\usepackage{newunicodechar}
\usepackage{tikz}
\usepackage{mathtools}
\usepackage{tabularx}
\usepackage{blindtext}
\usepackage{hyperref}

\newcommand{\ds}{\displaystyle}
\newcommand{\defeq}{\mathrel{\mathop:}=}

\newcommand{\realstwo}{\mathbb{R}^2}
\newcommand{\realsthree}{\mathbb{R}^3}

\newcommand{\xb}{{\bf{x}}}

\newcommand{\Dn}{\partial_{\nu}}

\newcommand{\cE}{{\mathcal{E}}}

\newcommand{\R}{\mathbb{R}}

\newcommand{\lb}{ \langle}
\newcommand{\rb}{ \rangle}
\theoremstyle{plain}
\newtheorem{theorem}{Theorem}[section]
\newtheorem{lemma}[theorem]{Lemma}
\newtheorem{proposition}[theorem]{Proposition}

\newtheorem{corollary}[theorem]{Corollary}

\theoremstyle{remark}
\newtheorem{remark}{Remark}[section]
\numberwithin{equation}{section}
\numberwithin{theorem}{section}
\numberwithin{remark}{section}
\numberwithin{assumption}{section}
\numberwithin{condition}{section}
\begin{document}

\title{Strong Stabilization of a 3D Potential Flow \\ via a Weakly Damped von Karman Plate}
 \author{\small \begin{tabular}[t]{c@{\extracolsep{1em}}cc@{\extracolsep{1em}}}
         Abhishek Balakrishna & Irena Lasiecka & Justin T. Webster  \\ 
 \it UMBC ~~&~~ \it University of Memphis ~~&~~ \it UMBC \\ 
 \it Baltimore, MD& \it Memphis, TN &\it Baltimore, MD \\  
bala2@umbc.edu & lasiecka@memphis.edu & websterj@umbc.edu
\end{tabular}}

\maketitle

\begin{abstract}
\noindent  The elimination of  aeroelastic instability (resulting in sustained oscillations of bridges, buildings, airfoils) is a central engineering and design issue. Mathematically,  this translates to  {\it strong asymptotic stabilization} of a 3D flow by a 2D elastic structure.  The stabilization (convergence to the stationary set) of a aerodynamic wave-plate model is established here. A 3D potential flow on the  half-space has a spatially-bounded von Karman plate embedded in the boundary. The physical model, then, is a Neumann wave equation with low regularity of coupling conditions. Motivated on empirical observations, we examine if intrinsic {\em panel damping} can stabilize the {\em subsonic} flow-plate system to a stationary point. Several partial results have been established through partial regularization of the model. Without doing so, classical approaches attempting to treat the given wave boundary data have fallen short, owing to the failure of the Lopatinski condition (in the sense of Kreiss, Sakamoto) and the associated regularity defect of the hyperbolic Neumann mapping. Here, we operate on the panel model as in the engineering literature with {\it no regularization or modifications}; we completely resolve the question of stability by demonstrating that {\em weak plate damping} strongly stabilizes system trajectories. This is accomplished by microlocalizing the wave data (given by the plate) and observing  an ``anisotropic''  {\em a microlocal compensation} by the plate dynamics {\em precisely where  the regularity of the 3D wave is compromsed} (in the characteristic sector). Several additional stability results for both wave and plate subsystems are established to ``push''  strong stability of the plate onto the flow. 
\vskip.15cm

\noindent {\em Key Terms}:  fluid-structure interaction, nonlinear plates, hyperbolic regularity,  Neumann-Dirichlet hyperbolic map, aeroelasticity, feedback stabilization, microlocal analysis
\vskip.15cm
\noindent {\em 2010 AMS}: 74F10, 74K20, 76G25, 35B40, 35G25, 37L15 \end{abstract}

\section{Introduction}
In this paper we resolve the notorious problem of stability of a 3D flow  through a damped 2D plate, acting only on a  bounded portion  of the wave boundary. 
The flow domain is the 3D half-space $\mathbb R^3_+$ and the two systems are strongly coupled at the interface $\Omega \subset \subset \partial \mathbb R^3_+$. The wave equation is fed Neumann-type data, written in terms of the material derivative of the plate, which itself obeys a (semilinear) von Karman plate equation. 
In this scenario, we ask the question: {\bf Does stability of the plate to equilibrium translate to stabilizability to equilibria for the wave?}

This question is quite nontrivial---in fact it has been an open problem for many years \cite{LBC96,chuey} for at least two immediate reasons. First, the interaction is of hybrid type, and effective propagation of dissipation/stabilization from 2D to 3D dynamics is a challenging, and typically not viable. Secondly, the 3D wave dynamics in this problem have no dissipative feature, other than dispersion to infinity.

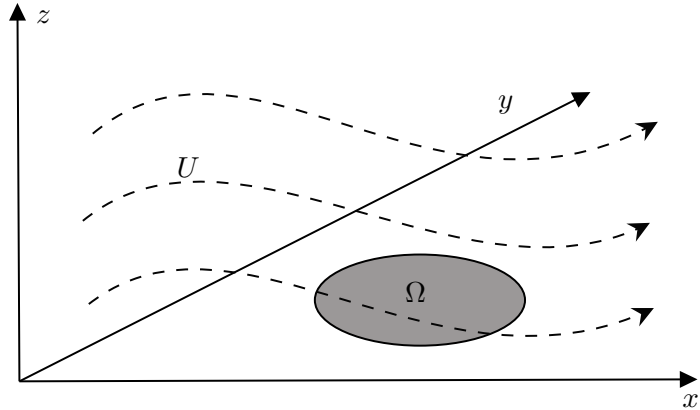
\begin{center}
	\tikzset{every picture/.style={line width=0.75pt}} 
	
	\begin{tikzpicture}[x=0.75pt,y=0.75pt,yscale=-1,xscale=1]
		
		\draw    (155,227) -- (494,226.01) ;
		\draw [shift={(497,226)}, rotate = 539.8299999999999] [fill={rgb, 255:red, 0; green, 0; blue, 0 }  ][line width=0.08]  [draw opacity=0] (8.93,-4.29) -- (0,0) -- (8.93,4.29) -- cycle    ;
		\draw    (155,227) -- (154.02,39) ;
		\draw [shift={(154,36)}, rotate = 449.7] [fill={rgb, 255:red, 0; green, 0; blue, 0 }  ][line width=0.08]  [draw opacity=0] (8.93,-4.29) -- (0,0) -- (8.93,4.29) -- cycle    ;
		\draw    (155,227) -- (440.32,82.36) ;
		\draw [shift={(443,81)}, rotate = 513.12] [fill={rgb, 255:red, 0; green, 0; blue, 0 }  ][line width=0.08]  [draw opacity=0] (8.93,-4.29) -- (0,0) -- (8.93,4.29) -- cycle    ;
		\draw  [fill={rgb, 255:red, 155; green, 152; blue, 152 }  ,fill opacity=1 ] (304,186) .. controls (304,173.3) and (327.73,163) .. (357,163) .. controls (386.27,163) and (410,173.3) .. (410,186) .. controls (410,198.7) and (386.27,209) .. (357,209) .. controls (327.73,209) and (304,198.7) .. (304,186) -- cycle ;
		\draw [color={rgb, 255:red, 5; green, 5; blue, 5 }  ,draw opacity=1 ] [dash pattern={on 4.5pt off 4.5pt}]  (192,102) .. controls (270.61,36.33) and (364.06,158.76) .. (475.32,96.95) ;
		\draw [shift={(477,96)}, rotate = 510.26] [fill={rgb, 255:red, 5; green, 5; blue, 5 }  ,fill opacity=1 ][line width=0.08]  [draw opacity=0] (10.72,-5.15) -- (0,0) -- (10.72,5.15) -- (7.12,0) -- cycle    ;
		\draw [color={rgb, 255:red, 5; green, 5; blue, 5 }  ,draw opacity=1 ] [dash pattern={on 4.5pt off 4.5pt}]  (187,146) .. controls (263.62,83.31) and (379.83,193.88) .. (471.62,147.71) ;
		\draw [shift={(473,147)}, rotate = 512.45] [fill={rgb, 255:red, 5; green, 5; blue, 5 }  ,fill opacity=1 ][line width=0.08]  [draw opacity=0] (10.72,-5.15) -- (0,0) -- (10.72,5.15) -- (7.12,0) -- cycle    ;
		\draw [color={rgb, 255:red, 4; green, 4; blue, 4 }  ,draw opacity=1 ] [dash pattern={on 4.5pt off 4.5pt}]  (190,188) .. controls (265.62,129.29) and (365.99,240.87) .. (473.38,190.77) ;
		\draw [shift={(475,190)}, rotate = 514.29] [fill={rgb, 255:red, 4; green, 4; blue, 4 }  ,fill opacity=1 ][line width=0.08]  [draw opacity=0] (10.72,-5.15) -- (0,0) -- (10.72,5.15) -- (7.12,0) -- cycle    ;
		
		\draw (488,232.4) node [anchor=north west][inner sep=0.75pt]    {$x$};
		\draw (395,81.4) node [anchor=north west][inner sep=0.75pt]    {$y$};
		\draw (162,38.4) node [anchor=north west][inner sep=0.75pt]    {$z$};
		\draw (348,175.4) node [anchor=north west][inner sep=0.75pt]    {$\Omega $};
		\draw (233.03,112.57) node [anchor=north west][inner sep=0.75pt]  [rotate=-0.3]  {$U$};
	\end{tikzpicture}
\end{center}

The model of interest has been prevalent in the aeroelasticity literature since the 1950s; see \cite{bolotin, dowellnon,B,dowell1} for engineering perspectives, and \cite{book,LBC96,chuey,bal} for more recent mathematical perspectives. We discuss the applied nature (and associated physical phenomenon) in detail below. For now, consider:
\begin{equation}\label{sys}\begin{cases}
		u_{tt}+k_0u_t+\Delta^2u+f_{V}(u)= p_0+\big[\partial_t+U\partial_x\big]\phi\big|_{\Omega} & \text { in }~~ \Omega\times (0,T),\\
		u=\Dn u = 0 & \text{ on } ~~\partial\Omega\times (0,T),\\
		(\partial_t+U\partial_x)^2\phi=\Delta \phi & \text { in }~~ \realsthree_+ \times (0,T),\\
		\partial_z \phi = (\partial_t+U\partial_x)u& \text{ on } ~~\Omega \times (0,T),\\
		\partial_z \phi = 0 & \text{ on } ~~{\overline{\Omega}}^c\times (0,T).
	\end{cases}
\end{equation} 
Above, the term $f_V$ represents a plate semilinearity (of {\em von Karman} type), and the nonlinear plate is coupled to a 3D inviscid potential flow. The quantity $U \in \mathbb R$ is the (normalized) unperturbed flow velocity in the $x$-direction, with $|U|<1$ indicating that a flow is  {\em subsonic}. The dynamic pressure drives the plate through the aeroelastic potential \cite{balshub,bal,dowellnon} and the static pressure $p_0(\mathbf x)$: \begin{equation} \label{pressure} p(\mathbf x,t)=p_0+(\partial_t+U\partial_x)\phi\big|_{\Omega}.\end{equation} 
Our goal for the above model---with some damping $k_0>0$ and subsonic flow $|U|<1$---is to achieve strong stabilization to equilibria of the entire system. 

This is a challenging problem for a variety of reasons. The principal issue involves the embedded 3D Neumann (perturbed) wave equation, for which the uniform Lopatinski condition does not hold \cite{sakamoto,miyatake}; this amounts to a critical loss of regularity in lifting from boundary data to the wave solution on $\mathbb R^3_+$. This, of course, is central  here, as we must translate stability of the plate (obtained from the weak damping) through the Neumann  solver to the hyperbolic flow. Indeed, the given wave data lacks sufficient regularity to decouple the problem whilst working with finite energy solutions.  For nearly 30 years, this has been the central hindrance in working on this and related systems for both well-posedness and stability \cite{chuey,LBC96}. However, a microlocal approach (of Tataru \cite{tataru}) can be utilized to address non-Lopatinski models by sharply characterizing the loss of Neumann regularity.

Direct stabilization of the flow via the plate must contend with this loss, and classical approaches such transform methods and Kirchhoff solution representations have fallen short. Indeed, any PDE-stabilization approach based on a global analysis of wave propagation is bound to fail.
The work in this paper, on the other hand, takes an alternate approach: {\bf we perform a microlocalization of the data itself, which, for the first time fully resolves the issue of stabilizability of the whole system}.
Treating the data $u_t+Uu_x$ globally, finite energy considerations permit only $L^2(\Omega \times (0,T))$ regularity, rather than the requisite $H^{1/3}(\Omega \times (0,T))$ \cite{tataru}. However, the data here also satisfy a PDE that can itself be microlocalized. {\em The main idea, then, is to determine whether the loss of regularity for the wave equation can be ``contained'' to a microlocal sector where the plate data itself is microlocally compensating}. The answer, is yes, but the analysis requires subtle and technical considerations along the way. In particular, one must work both classically and microlocally on the wave-plate system, as well as a reduced flow-plate system with memory, where the plate dynamics are shown to have some attractiveness and compactness properties \cite{delay} (from the dynamical systems point of view \cite{springer,cdelay1}). Invoking these properties, and ``propagating'' them to non-Lopatinski microlocal sector permits global estimates that yield strong stabilization of the wave dynamics, resulting in stabilization to equilibria of the total flow-plate interaction.

\begin{remark} We note that, in general, $|U|>1$ is permitted above and represents a viable mathematical model for supersonic flows \cite{supersonic,book}. This is not the focus here, as the supersonic panel system is not expected to have stationary end behavior \cite{supersonic,B,HHWW}. \end{remark}

\subsection{Physical Stabilization Problem}
Model \eqref{sys} is referred to a {\em panel flutter} model, in that it is used  to capture the {\em limit cycle oscillations} (LCOs) of a post-flutter panel in a potential flow of sufficient velocity \cite{bolotin,dowellnon,B,dowellrecent}.  The onset of this instability emanates from the linear portion of the coupled flow-plate model. The presence of the  {\it nonlinear} elastic restoring force ensures that dynamics remain globally bounded \cite{book,springer}, resulting in LCOs when the problem is linearly unstable due to the presence of $U \neq 0$. The model \eqref{sys} also captures aeroelastic buckling---a static (nonlinear) instability \cite{bolotin,dowellnon,B}.

The dynamics of this system have been studied in the PDE literature over the past 30 years, including well-posedness \cite{b-c,b-c-1,LBC96,webster,supersonic}, attractors (and dynamical systems) \cite{delay,webster2,Abhishek,springer,cdelay1,cdelay2,cdelay3}, and asymptotic stability/feedback control \cite{chuey,springer,Abhishek,ryz,ryz2,eliminating1,eliminating2}. We also note the foundational mathematical work on the dynamical systems properties of 1D flutter systems in \cite{marsden1,marsden2,cdelay1,balshub,shubov}. Of all of these works one challenge, motivated by an empirical engineering observation, has remained elusive. On that topic, renowned aeroelastician, E. Dowell,  states \cite{dowellrecent}: \begin{quote} ...subsonic panel flutter is unusual in aerospace applications...{\em if the trailing edge, as well as the leading edge, of the panel is fixed, then divergence (a static aeroelastic instability) will occur rather than flutter}. \end{quote} Mathematically: {\bf Can one show---directly from these PDEs---that the end behavior of  panel dynamics is static if $|U|<1$?} 
\begin{remark} We note that in the case of supersonic flows, panels may certainly flutter \cite{book, B,dowellrecent,HHWW,hlw}. Additionally, if a large portion of the leading or trailing edge are free, structures may flutter  {\em even for low flow velocities} (small $U$) \cite{HHWW,htw,dowellrecent}.\end{remark}

Established results on long-time behavior of the flow-plate system in \eqref{sys} provide a global compact attractor {\em for the plate component of the dynamics} for any value of $U\neq \pm 1$ without any imposed mechanical damping \cite{delay}. In the presence of an additional---albeit weak---structural damping, can one stabilize the entire system? {This is precisely the result we obtain}, indicating that subsonic instability can only be static in nature---{\em divergence}, a form of aeroelastic buckling \cite{bolotin,dowellnon}. From a mathematical perspective, {\bf we achieve strong stabilization to a non-trivial equilibrium} solution \cite{haraux1,haraux2}, where: (1) the flow dynamics take place on an unbounded domain; (2) non-positive conserved quantities must be accommodated; (3) damping is weak and takes place only on a compact subset of the flow boundary $\mathbb R^2$.

There have been several attempts to show such a stabilization result. Ultimately, to circumvent the fundamental challenges  associated to the Neumann-wave equation embedded in these dynamics \eqref{sys}$_{3}$--\eqref{sys}$_{4}$, model modifications and regularizations have been utilized. Those results are described in detail below in Section \ref{mainresults}. Here, the only addition to the classical model \cite{bolotin,dowellnon} is weak structural damping \cite{haraux1}. We believe the requirement of {\em some} weak and arbitrarily small damping in the plate is  only realistic,  since all elastic materials have a degree of internal damping. (See Remark \ref{dampers} below.)

Another central difficulty in achieving strong (rather than just weak) stability is the nonlocal character of elastic nonlinearity $f_V$---a lack of the so called unique continuation property. Rather, the game-changer in the present analysis is due to a more refined {\it linear analysis} of the built-in Neumann-to-Dirichlet (N-to-D) hyperbolic map here. The compromised regularity of this map in the non-Lopatinski case is compensated for by  taking full advantage of the physics of the model. Namely, this wave-plate interaction compensates microlocally for the loss of fractional derivatives in the N-to-D mapping. This provides sufficiently strong estimates for the full flow-structure system {\em without direct decoupling}, unlike what has been used in previous attempts in the literature \cite{Abhishek,ryz,ryz2}. 
\begin{remark}\label{dampers}
There are several relevant structural damping models, each of which have been broadly considered in the context of stability and attractors for semilinear plates of the form 
$$(1-\alpha \Delta) u_{tt}+\Delta^2u+f(u)+\mathbf Du_t=f.$$ The parameter $\alpha\ge0$ represents rotational inertia in the filaments of the plate \cite{lagnese}, and is taken to be zero in this treatment, in line with original then plate models appearing in the engineering literature \cite{bolotin,dowellnon}. We include this remark in order to comment on the relation of damping type to inertia. The inclusion of $-\alpha \Delta u_{tt}$ in the model for $\alpha >0$  boosts $u_t$ by one Sobolev unit (for finite energy solutions). This would resolve several aforementioned  regularity issues, most notably the regularity of hyperbolic N-to-D map described above, which is the main antagonist in this work. On the other hand, $\alpha > 0$ hinders the transfer of stability effects (dispersion) from the {\em flow to the plate}, precluding the existence of the plate attractor without additional {\em stronger} damping \cite{delay}; this would disqualify a key component of our analysis here. Thus $\alpha >0 $ is not a good model with respect to a stability analysis, both from the mathematical and physical points of view. 
 \vskip.2cm\noindent
 Weak (frictional) damping: $\mathbf D=k_0\mathbf I$. This damping remains at the level of the kinetic energy (KE) for the $\alpha=0$ case. If $\alpha>0$, this strength of damping is insufficient to control the KE.
\noindent \vskip.2cm\noindent
Square-root type damping: $\mathbf D=-k_1\Delta$. This damping is at the level of the square root of the biharmonic operator. This is the natural strength of damping required to stabilize the velocity/KE when $\alpha>0$.  This damping is {\em just} strong enough to render plate the dynamics parabolic \cite{redbook}.
\noindent \vskip.2cm\noindent
Strong (Kelvin-Voigt) damping: $\mathbf D=k_2\Delta^2$. This damping is at the level of the biharmonic operator. One can observe that the inclusion of this damping boosts the velocity by two Sobolev units and makes the dynamics parabolic \cite{redbook} (and references therein).
\end{remark}

\subsection{Precise Mathematical Model}
We consider a linear, inviscid potential flow as a perturbation of constant flow of velocity $U$ in the $x_1$-direction. Relevant flow parameters are scaled so  $U=1$ corresponds to the local speed of sound (Mach 1).  The flow resides in $\realsthree_+ = \{\xb \in\realsthree\, :\, x_3> 0\}$, and we consider a bounded domain $\Omega$ embedded in $\partial \mathbb R^3_+$. The set $\Omega$ has smooth boundary $\Gamma$ (and associated unit outward normal $\nu$), and we  take $\Omega$ to represent the plate's center plane (w.r.t. $x_3$) at equilibrium (as is standard in plate theory). 
In this situation, $u: \Omega \times [0,\infty) \to \mathbb R$ corresponds to the transverse $(x_3)$ plate deflections; $\phi: \mathbb R^3_+ \times [0,\infty) \to \mathbb R$ is the perturbation velocity potential so the potential flow velocity $\mathbf v$ is given by $\mathbf v = U\mathbf e_1+\nabla \phi$.\footnote{The flow model itself is obtained through linearization of the non-rotational, compressible and barotropic Euler equation about $U\mathbf e_1$.} The evolution system of interest is:
\begin{equation}\label{flowplate}\begin{cases}
		u_{tt}+\Delta^2u+k_0u_t +f_V(u)= p_0+r_{\Omega}\gamma_0\big[(\partial_t+U\partial_{x_1})\phi \big]& \text { in }~~ \Omega\times (0,T),\\
		u(0)=u_0,~~u_t(0)=u_1 & \text{ in }~~\Omega,\\
		u=\Dn u = 0 & \text{ on } ~~\partial\Omega\times (0,T),\\
		(\partial_t+U\partial_{x_1})^2\phi=\Delta \phi & \text { in }~~ \realsthree_+ \times (0,T),\\
		\phi(0)=\phi_0,~~\phi_t(0)=\phi_1 & \text { in }~~ \realsthree_+\\
		\partial_{x_3} \phi = [(\partial_t+U\partial_{x_1})u]_{\text{ext}}& \text{ on } ~~\Omega \times (0,T).
	\end{cases}
\end{equation} 
Above, the biharmonic expression $\Delta^2=[\Delta_{x_1,x_2}]^2$ in \eqref{flowplate}$_1$ is the iterated 2D Laplacian, while in the flow equation \eqref{flowplate}$_4$, $\Delta=\Delta_{x_1,x_2,x_3}$ is the 3D Laplacian. The vector $-\mathbf e_3$ is the unit outward normal for $\mathbb R^3_+$ (so $\partial_{\mathbf n} = -\partial_{x_3}$).
The notation $[\cdot]_{\text{ext}}$ above means extension by zero from $\Omega \to \mathbb R^2$ with corresponding restriction $r_{\Omega}[\cdot]$. The standard (Dirichlet) trace operator is denoted by $\gamma_0[\cdot]$ onto $\Gamma$ or $\{x_3=0\}$, of course in the appropriate functional sense. 
 Above, parameters such as mass, density, thickness, and stiffness have been scaled out. The remaining parameters are those relevant to this mathematical analysis: $|U| \in [0,1)$, as described above, and $k_0 \ge 0$ as the frictional/weak damping coefficient. The function $p_0 \in L^2(\Omega)$ is a stationary pressure on the top surface of the plate.

The (scalar) von Karman nonlinearity $f_V$ \cite{lagnese,ciarlet,springer} is given through the von Karman bracket and the Airy stress function. The {\em von Karman bracket} is \begin{equation} \label{bracket} [u,w]=u_{x_1x_1}w_{x_2x_2}+u_{x_2x_2}w_{x_1x_1}-2u_{x_1x_2}w_{x_1x_2},\end{equation}
while the Airy function is defined as an elliptic solver, namely, $v=v(u)$ is the solution to
\begin{equation}
\Delta^2 v = -[u,u]~~  \text{in}~~\Omega~~,~~
v=\partial_{\nu}v=0 ~~ \text{on}~~\Gamma.
\end{equation}
Finally, letting $F_0(\mathbf x)$ represent a stationary planar force on $\Omega$ (pre-stressing), we have {\em the von Karman nonlinearity}
~$ f_V(u)=-[u,v(u)+F_0].$~ We will take $F_0 \in H^4(\Omega)$, and describe the functional properties of $f_V$ below.

We utilize the standard notation and conventions for $L^p(\mathscr O)$ spaces and Sobolev spaces of order $s \in \mathbb R$, $H^s(\mathscr O)$ where $\mathscr O$ is some regular domain. The space $H_0^s(\Omega)$ denotes the completion of the test functions $C_0^{\infty}(\Omega)$ in the $H^s(\Omega)$ norm, with dual $[H_0^s(\Omega)]'=H^{-s}(\Omega)$. We will denote $||\cdot||_{H^s(\mathscr O)} = ||\cdot||_s$, where the spatial domain will be clear from context; we will identify $||\cdot||_{L^2(\mathscr O)}=||\cdot||$, omitting $s=0$. Inner products on $\mathbb R^3_+$ will be denoted by $(\cdot,\cdot) := (\cdot,\cdot)_{L^2(\mathbb R_+^3)}$ and on $\partial \mathbb R^3_+$ we utilize the notation $\langle \cdot,\cdot\rangle := (\cdot,\cdot)_{L^2(\Omega)}.$ The Dirichlet trace operator that we utilize is continuous and surjective $\gamma_0[\cdot]: H^1(\mathscr O)\to H^{1/2}(\partial \mathscr O)$. We denote an open ball of radius $R$ in a Banach space $X$ by $B_R(X)$. We freely identify $\mathbf x = (x_1,x_2,x_3)=(x,y,z)$, using whichever notation is most convenient in exposition.

\subsection{Energies and Functional Setup}

We introduce $W_k(\realsthree_+)$ as the homogeneous Sobolev space of the form $$W_k(\mathbb R^3_+) \equiv \Big\{\phi(\mathbf x) \in L^2_{loc}(\realsthree_+) ~:~ ||\phi||^2_{W_k}\equiv \sum_{j=0}^{k-1}\|\nabla \phi \|^2_{j,\realsthree_+} \Big\}.$$ In this way, the natural space associated with the flow potential variable $\phi$ is $W_1(\mathbb R^3_+)$ (i.e., this is the subspace of $L^2_{loc}(\mathbb R^3_+)$ taken with the {\em gradient norm} $||\nabla \phi||_{0,\realsthree_+}$). For the plate, the natural space associated to clamped boundary conditions is $H_0^2(\Omega)$.

In order to work in a dynamical systems framework, we take the principal state space to be $$Y = Y_{fl}\times Y_{pl} \equiv \big(W_1(\realsthree_+) \times L^2(\realsthree_+)\big)\times\big(H_0^2(\Omega) \times L^2(\Omega)\big),$$
 As seen in \cite{webster,supersonic,book}, we can consider a stronger space {\em on finite time intervals}: $$Y_s \equiv H^1(\realsthree_+)\times L^2(\realsthree_+) \times  H_0^2(\Omega) \times L^2(\Omega).$$

As we are only considering $|U| \in [0,1)$, flow terms involving $U\partial_x$ constitute a perturbation of the wave dynamics on $\realsthree_+$. Thus, the flow multiplier $\phi_t$ and plate multiplier $u_t$ yield:
\begin{align}
E_{pl}(u) = & ~\dfrac{1}{2}\big[||u_t||_{0,\Omega}^2+ ||\Delta u ||_{0,\Omega}^2+\frac{1}{2}||\Delta v(u)||_{0,\Omega}^2\big]-\langle F_0,[u,u]\rangle_{\Omega}+\langle p_0,u\rangle_{\Omega} \\
E_{fl}(\phi) =  &~\dfrac{1}{2}\big[||\phi_t||^2_{0,\realsthree_+}+||\nabla \phi||_{0,\realsthree_+}^2-U^2||\phi_x||^2_{0,\realsthree_+}\big],~~\\ E_{int} (u,\phi)= &  ~2U\langle\phi,u_x\rangle_{\Omega};~~\\[.1cm] \mathcal E(u,\phi) =&~ E_{pl}(u)+E_{fl}(\phi)+E_{int}(u,\phi),
\end{align}
where $E_{int}$ represents the non-positive, but conserved, {\em interactive} energy which will be addressed below. For the clamped plate, we invoke the equivalence of the full $H^2(\Omega)$ norm and $||\Delta \cdot ||_{0,\Omega}$. 

The formal energy identity (later obtained rigorously in Theorem \ref{nonlinearsolution}) is
\begin{equation}\label{formal} \mathcal E(t) +k_0\int_0^t||u_t||_0^2d\tau = \mathcal E(0).\end{equation} 
\begin{remark}  $E_{int}$ is unsigned  so it is not clear from \eqref{formal} that trajectories are time-bounded. \end{remark}

 \subsection{Definition of Solutions}
 A pair of functions $(u, \phi)$ such that 
\begin{equation}\label{platereq}
u \in C([0,T]; H_0^2(\Omega))\cap C^1([0,T];L^2(\Omega)),~
\phi \in C([0,T]; H^1(\realsthree_+))\cap C^1([0,T];L^2(\realsthree_+)) \end{equation} is said to be a {\em strong solution} to \eqref{flowplate} on $[0,T]$ if:
\begin{itemize}
\item $(\phi_t,u_t) \in L^1(a,b; H^1(\realsthree_+)\times H_0^2(\Omega))$ for any $(a,b) \subset [0,T]$.
\item $(\phi_{tt},u_{tt}) \in L^1(a,b; L^2(\realsthree_+)\times L^2(\Omega))$ for any $(a,b) \subset [0,T]$.
\item $\phi(t)\in H^2(\realsthree_+)$ and $\Delta^2 u(t) \in L^2(\Omega)$ for a.e. $t\in [0,T]$.
\item The equation ~~$u_{tt}+k_0u_t+\Delta^2 u +f_V(u)=p(\xb,t)$ holds in $L^2(\Omega)$ for a.e. $t >0$.
\item The equation ~~$(\partial_t + U\partial_x)^2\phi=\Delta \phi$ holds in $L^2(\mathbb R^3_+)$ and a.e. $t>0$.
\item The boundary conditions in \eqref{flowplate} hold for a.e. $t\in [0,T]$ and a.e. $\xb \in \partial \Omega$, $\xb \in \realstwo$ respectively. 
\item The initial conditions $\phi(0)=\phi_0, ~\phi_t(0)=\phi_1, ~u(0)=u_0, ~u_t(0)=u_1$ hold a.e. $\xb$.
\end{itemize}

A pair of functions $(\phi,u)$ is said to be a {\em generalized solution} to  problem (\ref{flowplate}) on  $[0,T]$ if there exists a sequence of strong solutions $(\phi^n(t);u^n(t))$ each with data $(\phi^n_0,\phi^n_1; u^n_0; u^n_1)$ such that
$$\lim_{n\to \infty} \max_{t \in [0,T]} \Big\{||\partial_t\phi-\partial_t \phi^n(t)||_{L^2(\realsthree_+)}+||\phi(t)-\phi^n(t)||_{H^1(\realsthree_+)}\Big\}=0$$ and
$$\lim_{n \to \infty} \max_{t \in [0,T]} \Big\{||\partial_t u(t)-\partial_t u^n(t)||_{L^2(\Omega)} + ||u(t) - u^n(t)||_{H_0^2(\Omega)}\Big\}=0.$$

Lastly,  a pair of functions $(\phi,u)$ such that 
\begin{align*}
u(\xb,t) \in \mathscr{W}_T\equiv \big\{ u \in L^{\infty}\big(0,T;H_0^2(\Omega)\big),~\partial_t u(\xb,t)\in L^{\infty}\big(0,T;L^2(\Omega)\big)\big\},\\
\phi(\xb,t) \in \mathscr{V}_T\equiv \big\{ \phi \in L^{\infty}\big(0,T;H^1(\realsthree_+)\big),~\partial_t \phi(\xb,t)\in L^{\infty}\big(0,T;L^2(\realsthree_+)\big)\big\},
\end{align*} is said to be a {\em weak solution} to \eqref{flowplate} on $[0,T]$ if:
\begin{itemize}
\item The initial conditions $\phi(0)=\phi_0, ~\phi_t(0)=\phi_1, ~u(0)=u_0, ~u_t(0)=u_1$ hold a.e. $\xb$.
\item$\ds 
\int_0^T\Big[\big\langle\partial_t u(t),\partial_t w(t)\big\rangle_{L^2(\Omega)}-\big\langle\Delta u(t),\Delta w(t)\big\rangle_{L^2(\Omega)}
- \big\langle f(u(t))-p_0,w(t)\big\rangle_{L^2(\Omega)}\\ \phantom{\hskip4cm}-\big\langle \gamma_0[\phi(t)],\partial_t w(t)+U\partial_x w(t)\big\rangle_{L^2(\Omega)}\Big]~dt
=\big\langle u_1-\gamma_0[\phi_0],w(0)\big\rangle_{L^2(\Omega)}$
 \newline for all test functions $w\in \mathscr{W}_T$ with $w(T)=0$. 
\item $\ds \int_0^T \Big[(\partial_t+U\partial_x)\phi(t),(\partial_t+U\partial_x)\psi(t)\big)_{L^2(\realsthree_+)}-\big( \nabla \phi(t),\nabla \psi(t)\big)_{L^2(\realsthree_+)}\\ \phantom{\hskip4cm}+\big((\partial_t+U\partial_x)u(t),\gamma[\psi(t)]\big)_{L^2(\Omega)}\Big]dt
=\big(\phi_1+U\partial_x \phi_0,\psi(0)\big)_{L^2(\realsthree_+)}$ \newline for all test functions $\psi\in\mathscr{V}_T$ such that $\psi(T)=0$. 
\end{itemize}

From the semigroup point of view, generalized solutions will correspond to semigroup solutions for an initial datum in $Y$.  It has been shown that  the nonlinearity $f_V$ (and other related nonlinearities \cite{supersonic,delay}) is locally Lipschitz from $H^2_0(\Omega)$ into $L^2(\Omega)$---see \cite{FHLT,FHLT1} and Lemma \ref{l:airy-1}. As a consequence  it is straightforward to see that generalized solutions are in fact {\em weak}. This weak solution is in fact unique (see \cite{FHLT,FHLT1} and  \cite[Chapter 6]{springer}).

\subsection{Well-Posedness of Solutions and Dynamical Systems}
This result follows from \cite{webster}. See also \cite{book,survey1,supersonic,eliminating1}.
\begin{theorem}[{\bf Nonlinear Semigroup}]
\label{nonlinearsolution} Assume $|U|<1$, $p_0 \in L^2(\Omega)$ and $F_0 \in H^{4}(\Omega)$.  Then for any $~T>0$,  \eqref{flowplate} has a unique strong (resp. generalized---and hence weak) solution on $[0,T]$, denoted by $S_t (y_0) $, for initial data $y_0=(\phi_0,\phi_1;u_0,u_1) \in Y_s$. (In the case of strong solutions, the natural compatibility condition must be in force on the data ~$\partial_z \phi_0 = (u_1+Uu_{0x})_{\text{ext}}$.)
Moreover,  $(S_t,Y)$ is a dynamical system.
 Weak (hence generalized and strong) solutions satisfy 
 \begin{equation}\label{eident}\ds {\mathcal{E}}(t)+k_0\int_s^t ||u_t(\tau)||^2  d\tau= {\mathcal{E}}(s),~~\forall t \ge s \ge 0.\end{equation}
Moreover, the solution $S_t(y_0)$ is stable in the norm of ~$Y$, i.e.,  $\exists~~C(||y_0||_{Y})$ so that:
\begin{equation}\label{stableS}
 \|S_t (y_0)\|_Y \leq C \left(\|y_0\|_{Y}\right),~~\forall~t \ge 0.
 \end{equation}
In addition, the  semigroup $S_t$ is locally Lipschitz on $Y$ on finite time horizons:
\begin{equation}\label{lip}
||S_t(y_1) - S_t(y_2) ||_Y \leq C (R,T) ||y_1-y_2||_Y,~~ \forall ~||y_i||_Y \leq R,~~ t \leq T. 
\end{equation}
\end{theorem}

\subsection{Stationary Problem}

	Since our main result involves convergence to the set of stationary solutions, we provide a brief discussion thereof.
	The stationary problem associated to \eqref{flowplate} has the form:
\begin{equation}\label{static}
\begin{cases}
\Delta^2u+f_V(u)=p_0+Ur_{\Omega}\gamma_0[\partial_x \phi]& \xb \in \Omega\\
u=\Dn u= 0 & \xb \in \Gamma\\
\Delta \phi -U^2 \partial_{x_1}^2\phi=0 & \xb \in \realsthree_+\\
 \phi_{x_3} = U\partial_{x_1} u_{\text{ext}}  & \xb \in \partial \realsthree_+
\end{cases}
\end{equation}
This problem has been studied extensively, most recently in \cite[Section 6.5.6]{springer}. 
	A {\em weak solution to \eqref{static}} is defined as a pair $\left(\phi,u \right)\in  W_1(\realsthree_{+})\times H_0^2(\Omega)$ such that 
\[
\langle \Delta {u}, \Delta w\rangle_{\Omega}- \langle \left[ {u}, v({u}) + F_0 \right], w \rangle_{\Omega} +  U \langle \gamma \left[ {\phi} \right], \partial_{x_1} w \rangle_{\Omega}  = \langle p_0, w\rangle_{\Omega}
\]
and 
\[
(\nabla {\phi}, \nabla \psi)_{\mathbb{R}_{+}^3} - U^2 (\partial_{x_1} {\phi}, \partial_{x_1} \psi)_{\mathbb{R}_{+}^3} + U\langle \partial_{x_1} {u}, \gamma [\psi]\rangle_{\Omega} = 0,
\]
for all test functions $(\psi,w) \in W_1(\mathbb R^3_+)\times H_0^2(\Omega)$. 
We have the following theorem  \cite[Theorem 6.5.10]{springer}:
\begin{theorem}\label{statictheorem}
	Suppose $0 \le |U| <1$ with $p_0 \in L^2(\Omega)$ and $F_0 \in H^3(\Omega)$. Then a {\em weak} solution to \eqref{static} exists and satisfies the additional regularity property ~$(\phi,u) \in W_2(\realsthree_+)\times H^4(\Omega)\cap H_0^2(\Omega).$ Such solutions are extremal points of the potential energy functional $$P(u,\phi) = \frac{1}{2}\|\Delta u\|_{\Omega}^2+\Pi(u)+\frac{1}{2}\|\nabla \phi\|_{\realsthree_+}^2-\dfrac{U^2}{2}\|\partial_{x_1}\phi\|_{\realsthree_+}^2 + U\langle \partial_{x_1} u, \gamma_0[\phi]\rangle_{\Omega},$$ considered for arguments $\left(\phi,u \right)\in  W_1(\realsthree_{+})\times H_0^2(\Omega)$. \end{theorem}
	We denote  $\mathcal N$ as the set of stationary solutions above, and note that, in general, $\mathcal N$ has multiple elements \cite{springer,ciarlet}. 
 But, for given loads $F_0$ and $p_0$, the set of stationary solution is generically finite \cite[Theorem 1.5.7 and Remark 6.5.11]{springer}. 

\section{Main Results and Technical Discussion}\label{mainresults}
In this section we state the main results in this treatment, with some  more technical discussion of this paper in relation to what is already established in the literature. Below, we are clear to point out what contributions are novel in this manuscript---finally resolving the issue of strong stabilization---and what is invoked from previous work. The central references  are \cite{ryz,ryz2,eliminating1,eliminating2,springer,Abhishek}.
	\begin{theorem}\label{regresult}
Let $U \in (-1,1)$ and $k_0>0$.  Assume $p_0 \in L^2(\Omega)$ and $F_0 \in H^{4}(\Omega)$. Then any  solution $(\phi(t),\phi_t(t); u(t),u_t(t))$ to \eqref{flowplate} with 
 data
$
(\phi_0,\phi_1;u_0,u_1) \in Y
$
that are
 spatially localized in the flow component (there is $\rho_0>0$ so that: $|\xb|\ge \rho_0$ $\implies$ $\phi_0(\xb) = \phi_1(\xb)=0$) has that 
 \begin{align*}\lim_{t \to \infty} \inf_{(\hat u,\hat \phi) \in \mathcal N}\left\{\|u(t)-\hat u\|^2_{H_0^2(\Omega)}+\|u_t(t)\|^2_{L^2(\Omega)}+\|\phi(t)-\hat \phi\|_{W_1( K_{\rho} )}^2+\|\phi_t(t)\|^2_{L^2( K_{\rho} )} \right\}=0
 \end{align*} 
 for any   $\rho>0$, where 
 $K_\rho\equiv \{ \xb\in\R_+^3\, : |\xb| \le \rho\}$ and
 $\mathcal N$ is as above.\end{theorem}
If we make the further assumption that  $\mathcal N$ is an isolated set (e.g., finite), we have the corollary:
\begin{corollary}\label{improve} 
Let the hypotheses of Theorem \ref{regresult} be in force. Assume that we have chosen $p_0, F_0$ so that $\mathcal N$ is an isolated set; then for any  solution $(\phi(t),\phi_t(t); u(t),u_t(t))$ to \eqref{flowplate} (with localized flow data, as above), there exists $(\hat u, \hat \phi) \in \mathcal N$ such that
\begin{align*}\lim_{t \to \infty} \left\{\|u(t)-\hat u\|^2_{H_0^2(\Omega)}+\|u_t(t)\|^2_{L^2(\Omega)}+\|\phi(t)-\hat \phi\|_{W_1( K_{\rho} )}^2+\|\phi_t(t)\|^2_{L^2( K_{\rho} )} \right\}=0, ~\forall \rho>0.\end{align*}
\end{corollary}

\subsection{Central Analytical Challenges}
We now technically outline the mathematical challenges presented by this model and its stability analysis. Even from the point of view of well-posedness, matters have been open until quite recently. Early mathematical work utilized Galerkin procedures, as well as  the inclusion of regularizations (thermoelasticity \cite{ryz,ryz2} or  rotational inertia in beam filaments \cite{b-c,b-c-1,springer}). The central issue in both well-posedness and stability centers about the loss of solution regularity for the wave equation when taken with (only) $L^2(\Omega)$ boundary data---the hyperbolic Neumann map does not satisfy the (uniform) Lopatinski condition in dimensions higher than one \cite{sakamoto,miyatake}. Specifically,  without regularizing the plate equation in some way, the Neumann wave data for finite energy solutions is globally measured in $$\partial_z \phi |_{\mathbb R^2} = [u_t+Uu_x]_{\text{ext}} \in C([0,T];L^2(\Omega)).$$ It is established by now \cite{tataru,miyatake,miyatake2,sakamoto,redbook} that such regularity does not recover finite energy wave solutions, so $\phi \not\in C([0,T];Y_{fl})$ (see Remark \ref{tatarucomment}) if we simply take generic data of this regularity. Hence, earlier results modified the model to boost $u_t \in L^2(\Omega) \mapsto u_t \in H^1(\Omega)$, and then use classical hyperbolic theory \cite{sakamoto,miyatake} to treat the flow and plate in a decoupled manner. Indeed,  the Neumann data need only be $H^{1/2}(\mathbb R^2)$ for a viable {\em Neumann mapping}. More recent well-posedness results \cite{webster,supersonic,Abhishek} do not attempt to decouple the system, but treat the flow-plate system holistically and exploit   dissipativity (in some sense), as well as estimate via microlocal analysis  some unbounded trace terms  appearing in the construction of the nonlinear semigroup \cite{supersonic,book}. 

For stability, the issue is yet more complicated and quite interesting. With a lack of dissipation for the full system, and the lack of compactness of the resolvent, one cannot rely on global approaches from dynamical systems (for instance, trying to conclude strong stability  based on a gradient system structure or clear asymptotic compactness).  In fact, the overall system has a Lyapunov function, {\em but it is not clear that it is strict} unless $U =0$.  As a consequence,  we cannot say the full system is {\em gradient} \cite{springer}.    As will be shown below, the existence of a compact global attractor for the plate  (Theorem \ref{th:main2}) and the finiteness of the dissipation integral (Corollary \ref{dissint}) allow one to show that all trajectories have the property that the {\it plate} component stabilizes to a stationary point (Theorem \ref{convergenceprops}); this of course requires $k_0>0$. There is no way to circumvent the fact that the  convergence of the plate must be ``lifted'' to the flow for stabilization of the whole system; thus there is no way to avoid the analysis of the aforementioned Neumann mapping. The critical issue can be seen from the classical point-wise plate-to-flow formula for the flow velocity given by \eqref{thisonenow}. It is direct to lift the strong convergence of the plate to the flow {\em in a weak sense}---pushing derivatives on test functions. On the other hand, for the main result, we must improve convergence from {\it weak to strong} for the flow. In general, this is a challenging proposition, as we must to push 2D stability to 3D dynamics from the boundary {\em with no 3D damping effects} and no help from dynamical systems tools.

As we have mentioned, previous works address stabilization in modified circumstances:  \cite{springer,Abhishek} address stabilization (with the same conclusion here) when rotational inertia and square root-type damping are considered $\alpha,k_1>0$. The earlier works \cite{ryz,ryz2} address stablization in the presence of thermoelastic effects for the plate. The works \cite{eliminating1} looks at the model with {\em smooth initial data} and obtains the stabilization result, whereas \cite{eliminating2} consider a different semilinearity (Berger's plate) and critically exploits its specific mathematical structure to obtain the stabilization result  but with a {\it large weak damping  coefficient}. 
In view of the above, our main contribution  is  that  {\em we obtain the main result for for finite energy solutions to the original, nonlocal, nonlinear model with no restrictions on the size of the damping coefficients.} Not surprisingly, this requires a novel approach to the problem.

\subsection{Key Observations and Proof Strategy}

Treating the boundary data $[u_t+Uu_x]_{\text{ext}}$ as a globally $L^2(\mathbb R^2)$ presents an insurmountable problem in the lack of regularity of the N-to-D  mapping.  Indeed, the loss of at least $1/3 $ derivative is intrinsic \cite{tataru}, unless the model is spatially 1D \cite{taylor}.  However, such a crude treatment neglects the microlocal properties of the data $u$  itself as a plate solution. Indeed, if we include the structure of the data in our microlocal analysis for the regularity of the plate-to-flow mapping we note that $u$ has additional microlocal regularity. {\bf This is main purpose of Section \ref{micro}, on which the entire breakthrough in this paper is based}. Indeed, the loss of a fraction of a derivative for the Neumann wave equation occurs in the so-called characteristic sector.\footnote{where the time and space dual variables are comparable, see Section \ref{flowsec}} However, in the characteristic sector, the plate has the property that $u_t \sim u_x$ by the comparability of time and space tangential derivatives therein.  Thus, since $u \in H^2(\Omega)$ globally, in the characteristic sector, we have $u_t \in H^1(\Omega)$ (of course, microlocally). Thus: {\bf the critical loss of regularity in the hyperbolic Neumann mapping is (more than) compensated for by the gain in regularity associated to the microlocalization of the boundary data}. We express this through a new microlocal estimate in Section \ref{micro} for the Neumann map, in the context of microlocally anisotropic spaces \cite{tataru}. 

\begin{remark} This is a remarkable compensation. If the plate were replaced by membrane, so the system was wave-wave coupling, no microlocal compensation in the 3D wave would present itself from the 2D wave in the characteristic sector. This would render the main result impossible.\end{remark}

In all, we are able to push the strong convergence of the plate to a stationary point to a strong convergence for the flow, improving from weak convergence of the flow (known) to strong convergence of the flow (to date, open). After identifying said  limit point as a stationary solution, one concludes strong convergence of the  flow-plate system  to  equilibria---hence the so called elimination of  flutter.   This is done in {\bf Section \ref{micro} and Section \ref{weaktostrong}, constituting  the main contribution of this work}.

 To establish the ``weak to strong'' result, it is critical, after obtaining the properties of the Neumann map in these $L^2$ based spaces (Section \ref{convmicro}), that we be able to pass directly  to point-wise information in time. This requires some work, since we are working with several different types of temporal regularity coming from the various components and techniques utilized. To do this, we use interpolation. We have uniform-in-time control of the plate on intervals of finite size (Theorem \ref{convergenceprops}, item 4.); this critically exploits the existence and compactness of an attractor for the plate dynamics \cite{delay}. 
 Finally, to close out the argument, we prove a ``converging together'' lemma on time intervals of uniform finite size (Section \ref{together}). We subtract off the known stationary state (from the established weak convergence) on a subsequence for the {\em linear} flow dynamics, invoke our obtained microlocal estimate in anisotropic spaces, and then obtain convergence to zero on translated time intervals. This allows us to relate point-wise flow information using the energy identity (in a decoupled sense), as it is linear. Arbitrariness of the time interval size allows us to show that flow is converging to a constant on these intervals, and the weak-to-strong improvement is {\em finally} deduced. 

We provide an Appendix where two other scenarios are briefly discussed (potential flow on a  half-space or bounded domain, $U=0$). We demonstrate in several ways how the main result here is {\em sharp}, from the point of view of the need for nonlinearity, the need for damping, and the non-expectation of uniform stabilization.

\section{Technical and Supporting Results}\label{techsec} 
We catalog several critical results that will be used in the proof of the main result. {\em Everything in this section has previously been used in the context of stability for this system, though we choose to include some proofs and  sketches here for the sake of self-containedness and easy referral for other arguments appearing herein.}
\subsection{Facts Concerning Plate Nonlinearity $f_V$}
As the von Karman nonlinearity is the principal engineering nonlinearity in the context of aeroelasticity \cite{bolotin, dowellnon, B, dowell1}, we require some facts to proceed. Indeed, $f_V$ has posed substantial challenges of regularity theory (and hence, uniqueness of solutions) until quite recently \cite{FHLT,springer}.  We present the key mathematical properties of the Airy stress function and $f_V$ that will be used in our work below.
Let us formally define the Airy stress function $v=v(u,w)$ in terms of \eqref{bracket} as the solution to 
$$\begin{cases} \Delta^2v=-[u,w] & \text{in}~~\Omega \\ u=w=\partial_{\nu}u=\partial_{\nu}w=0 & \text{on}~~\Gamma\end{cases}$$
Then, the so called {\em sharp regularity of Airy's stress function} is {\em critical} to all the results stated below.
 \begin{lemma}\label{l:airy-1}
 The function $v(u,w)$ as defined above satisfies:
 \begin{enumerate}
 \item
 $||v(u,w)||_{W^{2,\infty}(\Omega) } \leq C ||u||_2||w||_2 $ for all $u,w \in H^2(\Omega)$.
 \item
 The map ~$(u,w) \mapsto v(u,w) $ is locally Lipschitz from
 $H_0^2(\Omega) \times H_0^2(\Omega) \rightarrow W^{2,\infty}(\Omega)$.
 \end{enumerate}
 \end{lemma}
 \noindent This  implies $f_V(u)= - [u, v(u) + F_0 ]$ is locally Lipschitz~ $H_0^2(\Omega) \to L^2(\Omega)$ \cite[Section 1.4, pp.38--45]{springer}. 
 \begin{theorem}\label{nonest}
Let $u^i \in {B}_R(H^2_0(\Omega))$, $i=1,2$, and $z=u^1-u^2$.
 Then 
\begin{equation}\label{f-est-lip}
||f_V(u^1)-f_V(u^2)||_{-\delta}  \le C_{\delta}\big(1+||u^1||_2^2+||u^2||_2^2\big)||z||_{2-\delta} \le C(\delta,R)||z||_{2-\delta}~~~ for~ all ~~\delta \in [0,1].
\end{equation}\end{theorem}

The second critical property of the von Karman nonlinearity, in the context of long time behavior, is its ``good''  potential energy (namely,  control of low frequencies) \cite{springer}.
\begin{proposition}\label{epsilonlemma}  For any  $\eta,\epsilon > 0 $ there exists $M_{\epsilon,\eta} $ such that
	\begin{equation}\label{l:epsilon}\|u\|^2_{2-\eta} \leq \epsilon \big[\|\Delta u\|_0^2  + ||\Delta v(u)||_0^2 \big] + M_{\eta,\epsilon},~~\forall~u \in H^2(\Omega) \cap H_0^1(\Omega).\end{equation}
 \end{proposition}
\noindent The Lemma \eqref{l:epsilon} is obtained through a compactness uniqueness argument that exploits superlinearity of $f_V$ and a maximum principle associated to the Monge-Ampere equations. \begin{remark} The inequality in Proposition \ref{epsilonlemma} is valid for plates with clamped or hinged boundary conditions. However, {\it free} boundary conditions are not covered by this argument. \end{remark}

\subsection{Control of Energies}
In this section we remark that, in the norm of $Y$, solutions are global-in-time bounded. 
\begin{proposition}\label{littleest} For all $\phi \in W_1(\mathbb R_+^3)$:
	\begin{equation}\label{tracebound}||r_{\Omega}\gamma_0[\phi]||_{L^2(\Omega)} \le C_{\Omega}||\nabla \phi||_{L^2(\realsthree_+)}.\end{equation}
The interactive energy $E_{int}$, is then controlled in the following way:
	\begin{equation}\label{intbound} \big|E_{int}(u(t),\phi(t))\big| \le \delta \|\nabla \phi(t)\|_{\realsthree_+}^2+C(U,\delta)\|u_{x}(t)\|_{\Omega}^2, ~~\delta>0. \end{equation}
	
\end{proposition}
\noindent Estimate \eqref{tracebound} follows from the Hardy inequality, and from \eqref{tracebound}, Young's gives \eqref{intbound} \cite{springer,webster}.

Now, let us define the positive part of the energy $\mathcal E$ as:
\begin{equation}\label{posenergy}
 \mathcal E_*(t)=\dfrac{1}{2}\big[||u_t||_0^2(\Omega)+ ||\Delta u ||_0^2+\frac{1}{2}||\Delta v(u)||_0^2+||\phi_t||_0^2+||\nabla \phi||_0^2\big]\end{equation}
 Then, via Proposition \ref{littleest}, we obtain control of the unsigned energy by the positive part:
 \begin{lemma}\label{energybound} For generalized solutions to \eqref{flowplate}, there exist positive constants $c,C,$ and $M_{p_0,F_0} $ that are positive, and do not depend on the trajectory, such that:
	\begin{equation} \label{estim}
	c \mathcal E_*(t)-M_{p_0,F_0} \le \cE(t) \le C \mathcal E_*(t)+M_{p_0,F_0},
	\end{equation} \end{lemma}
\begin{remark} Even with these bounds, and the associated energy identity from Theorem \ref{nonlinearsolution}, it is not clear that $\cE_*$ provides a strict Lyapunov function for the dynamics (rendering them {\em gradient}). This is due to the specific structure of the coupling in the problem (involving material derivatives), and is one way to view one of the central challenges here with regard to stability. \end{remark}
With the energy identify in Theorem \ref{nonlinearsolution}, we may synthesize the above to obtain:
\begin{lemma}\label{globalbound}
Any generalized solution to \eqref{flowplate}, with $f(u)=f_V(u)$ satisfies the bound \begin{equation}
\sup_{t \ge 0} \left\{\|u_t\|_{\Omega}^2+\|\Delta u\|_{\Omega}^2+\|\phi_t\|_{\realsthree_+}^2+\|\nabla \phi\|_{\realsthree_+}^2 \right\}  \leq C\big(\|S_t(y_0)\|_Y\big)< + \infty.\end{equation}
\end{lemma}
\noindent The above inequalities are proven and discussed in detail in \cite{springer, webster}.

Finally, as a corollary to the energy identity \eqref{eident} and Lemma \ref{energybound}, we have:

\begin{corollary}\label{dissint}
Let $ k_0> 0 $. Then for a generalized solution we have $$  \int_0^{\infty} \|u_t(t)\|_{0,\Omega}^2 dt \leq  K_{u,k_0} < \infty.$$
\end{corollary}
\noindent The boundedness obtained from the energy identity---which does not hold for supersonic flows $|U|>1$---will be used critically to obtain convergence of trajectories to stationary points.

In order to describe the dynamics of the flow in the context of long-time behavior,  it is necessary to introduce local spaces $Y_{fl,\rho}$: $$\|(\phi_0,\phi_1)\|_{Y_{fl},\rho}\equiv \int_{ K_{\rho} } |\nabla \phi_0|^2  + |\phi_1|^2 d\xb,$$  
where we recall $K_{\rho} \equiv \{ \xb \in \mathbb R^3_{+}; |\xb |\leq \rho \} $. We denote by $Y_{\rho} \equiv Y_{fl,\rho} \times Y_{pl}$, and we will consider convergence (in time) in $Y_{\rho}$ for any $\rho>0$. 

Again, by virtue of the Hardy inequality \cite[p.301]{springer}
$$\|\phi_0\|^2_{L^2(K_{\rho})} \le C_{\rho}\|\nabla \phi_0\|^2_{L^2(\realsthree_+)}$$ and hence
$$\|(\phi_0,\phi_1)\|_{Y_{fl},\rho}^2 \le  \|(\phi_0,\phi_1)\|_{H^1(K_{\rho})\times L^2(K_{\rho})}^2\le C_{\rho} \|(\phi_0,\phi_1)\|_{Y_{fl}}^2.$$

\subsection{Neumann Formulae and Flow Decomposition}\label{flowsec}
	
	Let us consider the ``decoupled'' Neumann flow problem:
\begin{equation}\label{floweq*}
\begin{cases}
(\partial_t+U\partial_x)^2\phi=\Delta \phi & \text{ in }~\mathbb R_+^3\\
\partial_{\nu} \phi\Big|_{z=0} = h(\xb,t) & \text{ in }~\mathbb R^2\\
\phi(t_0)=\phi_0;~~\phi_t(t_0)=\phi_1
\end{cases}
\end{equation}
for which we have \cite{b-c-1,springer,miyatake}:
\begin{theorem}\label{flowpot}
Assume $U \in \mathbb R$, $U\neq \pm 1$; take $(\phi_0,\phi_1) \in H^1(\realsthree)\times L^2(\realsthree).$ If ~$h \in C\left([t_0,\infty);H^{1/2}(\mathbb R^2)\right)$ then \eqref{floweq*} is well-posed (in the weak sense) with 
$$\phi \in C\left([t_0,\infty);H^1(\realsthree_+)\right),~~\phi_t \in C\left( [t_0,\infty);L^2(\realsthree_+)\right).$$
\end{theorem}
\begin{remark}\label{tatarucomment}
Finite energy $H^1(\Omega) \times L^2(\Omega)$ solutions would be obtained if $h\in H^{1/3}((0,T) \times \mathbb R^2)$ \cite{tataru}. \end{remark}

A key issue we address here is in the lack of  sufficient regularity for $h=[u_t+Uu_x]_{\text{ext}}$ in the coupled system. This is to say that, for our coupled system \eqref{flowplate}, we have \begin{equation}\label{h}
h \equiv [u_t + U u_{x_1}]_{\text{ext}} \in C([0,T];L^2(\mathbb R^2)).
\end{equation}
So, as per \eqref{flowpot}, there is insufficient regularity of the data to consider decoupling the flow and plate,  and work within the framework of finite energy solutions as was done in previous work on stability \cite{b-c,b-c-1,Abhishek,ryz,ryz2}. On the other hand, we do have solutions to the full flow-plate system \eqref{flowplate} via our well-posedness result Theorem \ref{nonlinearsolution}.

Now, let us denote	
$\phi^*$ by the solution to \eqref{floweq*} with $h \equiv 0$, and $\phi^{**}$ as the solution to \eqref{floweq*} with $\phi_0=\phi_1\equiv 0$.
We may look at $\phi^*$ and $\phi^{**}$ separately via linearity of the flow equation. With  $(\phi_0,\phi_1) \in H^1(\R_+^3)\times  L^2(\R_+^3)$ one obtains \cite{supersonic,miyatake}:~~
 ~$\ds (\phi^*(t),  \phi^{*}_t(t)) \in   H^1(\R_+^3)\times  L^2(\R_+^3).$ Thus, by the established well-posedness in Theorem \ref{nonlinearsolution} for $\phi = \phi^*+\phi^{**}$
 we also have that ~$\ds 
 (\phi^{**}(t),  \phi^{**}_t(t)) \in  H^1(\R_+^3)\times  L^2(\R_+^3).
 $

 For the analysis of $\phi^*$ we use the tools developed in \cite{b-c,b-c-1}, namely, the Kirchhoff-type representation for the solution  $\phi^*(\xb,t)$
in $\R_+^3$  \cite[Theorem~6.6.12]{springer}. We  conclude that
if the initial data   $\phi_0$ and $\phi_1$ are  localized in the ball $K_{\rho_0} $,
then by Huygen's principle in 3D,
   one obtains for any $ \rho$ that  $\phi^*(\xb,t)\equiv 0$ for all $\xb\in  K_{ \rho} $
and $t\ge t_{\rho}$. 
Thus $\phi^*$ tends to zero (exponentially) in the sense of the local flow energy, i.e.,  \begin{equation}\label{starstable} \|\nabla \phi^*(t)\|_{L^2( K_{\rho} )}^2 + \|\phi^*_t(t)\|_{L^2( K_{ \rho} )} \to 0, ~~ t \to \infty,\end{equation} for any $ \rho>0$.
Also, in this case,
\begin{equation}\label{traceh}
\big(\partial_t+U\partial_{x_1}\big)\gamma[\phi^*]\equiv0,~~~\xb\in \Omega,~t\ge t_{\rho}.
\end{equation}

On the other hand, for $\phi^{**}$ as above, we have another representation. For the following calculations, we assume sufficient regularity, which will produce point-wise representations that will hold, at minimum, in distribution for weak solutions to \eqref{flowplate}. Let $H(\cdot) $ be the Heaviside function.
\begin{theorem}\label{flowformula}
	Let ~$\ds h(\xb,t) =[u_t (x_1,x_2,t)+Uu_{x_1}(x_1,x_2,t)]_{\text{ext}},$  then there exists a time $t^*(\Omega,U)$ such that, for all $t>t^*$, we have the following representation for a solution:
	\begin{equation}\label{phidef}
	\phi^{**}(\xb,t) = {-}\dfrac{H(t-x_3) }{2\pi}\int_{x_3}^{t}\int_0^{2\pi}(u^{\dag}_t(\xb,t,s,\theta)+Uu^{\dag}_{x_1}(\xb,t,s,\theta))d\theta ds.
	\end{equation}
	 where $$f^{\dag}(\xb,t,s,\theta)=f\left(x_1-\kappa_1(\theta,s,x_3), y-\kappa_2(\theta,s,x_3), t-s\right),~~\text{and}$$  $$\kappa_1(\theta,s,x_3)=Us+\sqrt{s^2-x_3^2}\sin \theta,~~\kappa_2(\theta,s,x_3) = \sqrt{s^2-x_3^2}\cos \theta.$$
	\end{theorem}
	Letting the time $t^*$ be given by:
	\begin{equation}\label{tstar}
	t^*=\inf \{ t~:~\xb(U,\theta, s) \notin \Omega \text{ for all } (x_1,x_2) \in \Omega, ~\theta \in [0,2\pi], \text{ and } s>t\}, ~\text{with}
	\end{equation}  \begin{equation}\label{xescape}\xb(U,\theta,s) = \left(x_1-(U+\sin \theta)s,x_2-s\cos\theta\right) \subset \realstwo\end{equation} (not to be confused with $\xb = (x_1,x_2)$),
	 partials may be computed directly (see \cite{Abhishek}) by waiting a sufficient time. The two most relevant are 
	\begin{align}\label{thisonenow}
\phi^{**}_{t}(\xb, t) = & ~\dfrac{1}{2\pi}\Big\{\int_0^{2\pi}u_{t}^{\dag}(\xb,t,t^*,\theta) d\theta -\int_0^{2\pi} u_{t}^{\dag}(\xb,t,x_3,\theta )d\theta \\\nonumber
&+U\int_{x_3}^{t^*}\int_0^{2\pi}[\partial_{x_1} u_{t}^{\dag}](\xb,t,s,\theta) d\theta ds+\int_z^{t^*}\int_0^{2\pi}\dfrac{s}{\sqrt{s^2-z^2}}[M_{\theta}u_{t}^{\dag}](\xb,t,s,\theta) d\theta ds\Big\}, \end{align} 
with $M_{\theta}:=\sin(\theta)\partial_{x_1}+\cos(\theta)\partial_{x_2}$ 
 and for $i=1,2$: {\small \begin{align} \label{partialder} 
\phi^{**}_{x_i}(x,t) =& ~\frac{1}{2\pi}\int_{x_3}^{t^*}\int_0^{2\pi} U\partial_{x_1}u_{x_i}^\dagger(x,t,s,\theta) d\theta ds +\frac{1}{2\pi}\int_{x_3}^{t^*}\int_0^{2\pi} \partial_tu_{x_i}^\dagger(x,t,s,\theta) d\theta ds. \\ \nonumber
 \partial_{x_3}\phi^{**}(x,t)=&~ (\partial_t+U\partial_{x_1})u(x_1-Ux_3,x_2,t-x_3)+\frac{1}{2\pi}\int_{x_3}^{t^*}\frac{x_3}{\sqrt{s^2-x_3^2}}\int_0^{2\pi} (\partial_t+U\partial_{x_1})[M_\theta u]^\dagger(x,t,s,\theta) d\theta. \end{align}}

From here, we can explicitly solve for the Dirichlet trace of the material derivative appearing in the plate equation in terms of the Neumann data $h=[u_t+Uu_{x_1}]_{\text{ext}}$.
	Considering the term $$r_{\Omega}\gamma_0\left[\left(\partial_t+U\partial_{x_1}\right) \phi\right]=r_{\Omega}\gamma_0\left[\left(\partial_t+U\partial_{x_1}\right) \phi^{**}\right]$$ for $t>t_{\rho_0}$ by \eqref{traceh}, where  $supp(\phi_0),~supp(\phi_1) \subset K_{\rho_0}$.  Using the above expressions for $\partial_t\phi^{**}$ \eqref{thisonenow} and $\partial_{x_1}\phi^{**}$ \eqref{partialder}, we obtain (after sufficiently long time): 
	\begin{equation}\label{delayt}
	r_{\Omega}\left(\partial_t+U\partial_{x_1}\right)\gamma_0\left[\phi \right]=-(\partial_t+U\partial_{x_1})u-q(u^t),
	\end{equation}
	for $t\geq \max\{t^*,t_{\rho_0}\}$ ($t^*$ as defined above in \eqref{tstar}) with
	\begin{equation}\label{qyou}
	q(u^t)=\dfrac{1}{2\pi}\int_0^{t^*}ds\int_0^{2\pi}M^2_{\theta}\left[u\right]_{ext}(\mathbf x(U,\theta,s)d\theta.
	\end{equation}
	The notation for $u^t$ indicates dependence on the entire set $\big\{u(t+s)~:~s \in (-t^*,0)\big\},$ where $t^*(\Omega,U)$ is the characteristic ``memory'' time given in \eqref{tstar} \cite{springer,cdelay1,cdelay2,cdelay3}.

\subsection{Plate Attractor}
From the point-wise formulate of the previous section---including the N-to-D map in \eqref{delayt}---we obtain the theorem below by waiting $t^{\#}=\max\{t^*,~t_{\rho}\}$.
\begin{theorem}\label{rewrite}
	Let $k_0 \ge 0$, $|U|<1$  and $(\phi_0,\phi_1; u_0,u_1) \in   Y_{\rho_0}$.  Then the there exists a time $t^{\#}(\rho_0,U,\Omega) > 0$ such that for all $t>t^{\#}$ a plate solution $u(t)$ to (\ref{flowplate}) satisfies:
	\begin{equation}\label{reducedplate}
	u_{tt}+\Delta^2u+k_0u_t+f_V(u)=p_0-(\partial_t+U\partial_{x_1})u-q(u^t)
	\end{equation}
and with $M_{\theta}$ and $t^*$ as in the previous section, we have:
	\begin{equation}\label{potential*}
	q(u^t)=\dfrac{1}{2\pi}\int_0^{t^*}ds\int_0^{2\pi}d\theta [M^2_{\theta} u_{\text{ext}}](x_1-(U+\sin \theta)s,x_2-s\cos \theta, t-s).
	\end{equation}
	The sense $u(\xb,t)$ satisfies \eqref{reducedplate} is the same as how $u(\xb,t)$ satisfies \eqref{flowplate} (strong, generalized, weak). 
\end{theorem}
\noindent With this {\em reduction} in hand, we can perform a stability analysis directly on this plate equation with memory; this is the main contribution of \cite{delay}, which provides:
 \begin{theorem}\label{th:main2}
Let $k_0 \ge 0$, $|U|<1$, with $F_0 \in H^{4}(\Omega)$ and $p_0 \in L^2(\Omega)$.
 Then there exists a compact set $\mathscr{U}\subset \subset Y_{pl}$ (in fact $\mathscr{U} \subset \big(H^4(\Omega)\cap H_0^2(\Omega)\big) \times H_0^2(\Omega)$) of finite fractal dimension such that $$\lim_{t\to\infty} d_{\mathcal H} \big( (u(t),u_t(t)),\mathscr U\big)=\lim_{t \to \infty}\inf_{(\nu_0,\nu_1) \in \mathscr U} \big( ||u(t)-\nu_0||_2^2+||u_t(t)-\nu_1||^2\big)=0$$
for any generalized solution $(\phi,\phi_t;u,u_t)$ to (\ref{flowplate}) with $(\phi_0,\phi_1; u_0,u_1) \in   Y_{\rho_0}$. 
\end{theorem}

The proof of this result follows from rewriting the  dynamical system  $(S_t, Y_{\rho})$ generated by  (\ref{flowplate}) as a delayed  dynamical system corresponding to solutions of (\ref{reducedplate}), which we denote as $(T_t,\mathbf H)$ with $\mathbf H=H_0^2(\Omega)\times L^2(\Omega) \times L^2(-t^*,0;H_0^2(\Omega))$. We operate on this reduced plate dynamical system, and show the existence of a compact global attractor $\mathbf A \subset \mathbf H$; we then take $\mathscr U$ to be the projection of $\mathbf{A}$ on $H_0^2(\Omega)\times L^2(\Omega)$, which yields Theorem \ref{th:main2}.

\subsection{Weak Convergence and Strong Plate Convergence}\label{convergencesec}

We begin by noting that the existence of the compact attracting set $\mathscr U$ for the plate dynamics in Theorem \ref{th:main2}  we infer that for any initial data $y_0=(\phi_0,\phi_1;u_0,u_1)\in Y_{\rho_0}$ and any sequence of times $t_n \to \infty$ there exists a subsequence of times $t_{n_k} \to \infty$ and a point $(\hat u, \hat w) \in Y_{pl}=H_0^2(\Omega)\times L^2(\Omega)$ such that $(u(t_{n_k}),u_t(t_{n_k}) \to (\hat u, \hat w)$ strongly in $Y_{pl}$. Indeed, by the definition of convergence in the Hausdorff distance of $(u(t),u_t(t)) \to \mathscr U$, and the compactness of the $\mathscr U$, we generate a convergence subsequence on $\mathscr U$. By uniqueness of limits, we obtain aforesaid point and subsequence.

Additionally, by the global-in-time bound given in Lemma \ref{globalbound}, we know that the set $\{S_{t}(y_0)\}$ is bounded in $Y$, and hence for any sequence $t_{n} \to \infty$ there exists a subsequence $t_{n_k}$, and a point $(\hat \phi, \hat \psi;\hat u, \hat w) \in  Y$ such that $S_{t_{n_k}}(y_0) \rightharpoonup  ( \hat \phi, \hat \psi;\hat u, \hat w)$ in $Y_{\rho}$ for any $\rho$. Utilizing these results in conjunction, for any sequence of times $t_{n} \to \infty$ there is a subsequence such that both of the above convergences hold simultaneously.
We  collect the above  convergences, and add two additional points which will be central to the arguments that follow.
\begin{theorem}\label{convergenceprops}
For any initial data $y_0=(\phi_0,\phi_1;u_0,u_1)\in Y_{\rho_0}$ and any sequence of times $t_n \to \infty$ there is $t_{n_k}$ and a point $\hat y = (\hat \phi, \hat \psi;\hat u, \hat w)$ so that for a  generalized solution $(\phi,\phi_t;u,u_t)$ we have:
 \begin{enumerate}
 \item $S_{t_{n_k}}(y_0) \rightharpoonup  \hat y$ in $Y_{\rho}$ for any $\rho>0$. \vskip.2cm
\item
$\| u(t_{n_k}) - \hat{u}\|_{2,\Omega} \rightarrow 0$. \vskip.2cm
\item $\|u_t(t)\|^2_{0} \to 0, ~t \to \infty$, and hence $\hat w =0$. \vskip.2cm
\item Along the sequence of times $t_{n_k} \to \infty$ with corresponding limit point $\hat u$, we have \begin{align}
\sup_{\tau \in [-c,c]}\|u(t_{n_k}+\tau)-\hat u\|_{2-\delta,\Omega} \to 0 ~\text{ for any fixed}~ c > 0, ~\delta\in (0,2).
\end{align}
\end{enumerate}
\end{theorem}
\begin{proof}[Proof of Theorem \ref{convergenceprops}]
{{Points (1.) and (2.) follow as described above, through compactness of the attracting set $\mathscr U \subset \subset Y_{pl}$ and bounded dynamics system dynamics (Theorem \ref{th:main2} and Lemma \ref{globalbound}, resp.).}

To establish point (3.), we will prove that, for a given generalized solution, $u_t(t) \to 0$ in $\mathcal D'(\Omega)$, which is to say that for any $\eta \in C_0^{\infty}(\Omega)$ we have
$$\lim_{t\to\infty} \langle u_t(t),\eta \rangle_{L^2(\Omega)} =0.$$
Given this fact,  and the strong subsequential limits given from the attractor, we will boost the pointwise-in-time convergence in $\mathcal D'(\Omega)$ to that of $L^2(\Omega)$, and identify all subsequential limits with zero. 

To show $u_t(t) \to 0 $ in $\mathcal D'(\Omega)$, we will operate on the quantities $\langle u_t(t),\eta \rangle_{\Omega}$ and its time derivative for a given $\eta \in C_0^{\infty}(\Omega)$.
 We first consider a {\em strong solution}, and show that the quantity $\partial_t\langle u_t(t),\eta \rangle_{\Omega}$ is uniformly bounded for $t \in [0,\infty)$. We note, from the plate equation itself in \eqref{flowplate}, that \begin{align*}
|\lb u_{tt},\eta  \rb_{\Omega}|= &~ \big|-k\lb u_t,\eta  \rb_{\Omega}-\lb\Delta  u, \Delta \eta  \rb_{\Omega}+\lb f_V(u), \eta  \rb_{\Omega}+\lb p_0,\eta  \rb_{\Omega} +\lb \gamma_0[\phi_t+U\phi_x],\eta  \rb_{\Omega} \big|\\
 \le &~ \|\eta \|_{2,\Omega}\big(E_{pl}(t)+C\big)+|\lb \gamma_0[\phi_t+U\phi_x],\eta \rb_{\Omega}|,
\end{align*} where all quantities on the RHS are well-defined as $L^2(\Omega)$ inner products for strong solutions. We can invoke our reduction result, 
since we have assumed $|U|<1$ and that there exists a $\rho_0$ such that for $|\xb|>\rho_0$,~~ $\phi_0(\xb)=\phi_1(\xb)=0$. Utilizing Theorem \ref{rewrite}, we note that for sufficiently large times $t>t^\#(\rho_0,\Omega,U)$ we may utilize the formula for the trace of the flow, namely:
$$\gamma_0[\phi_t+U\phi_x]=-(u_t+Uu_x)-q(u^t), ~~\xb \in \Omega.$$ It is then follows that for {\bf sufficiently large times} $t>t^\#$: 
$$|\lb \gamma_0[\phi_t+U\phi_x],\eta\rb_{\Omega}| \le \|\eta\|_{0,\Omega}\cdot \sup_{\tau \in [t^\#,\infty)}\left\{ \|\Delta u\|^2+\|u_t\|^2\right\} \le C\left(||(u,u_t)||_{Y_{pl}}\right)\|\eta\|,$$ where in the final step we have used the boundedness in Theorem \ref{globalbound}. For smooth solutions, we thus obtain a uniform-in-time bound (on the time derivative) of the form 
$$\left|\dfrac{d}{dt}\langle u_t ,\eta \rangle_{\Omega}\right| \le C\left(||(u,u_t)||_{Y_{pl}}\right)||\eta||_{2,\Omega}.$$
By density (and the definition of a generalized solution), we can extend this inequality to a generalized solution.
Hence the quantity $\partial_t\langle u_t(t),\eta \rangle_{\Omega}$ is bounded {\em uniformly} in $t$ for any generalized solution satisfying the hypotheses at hand.

Now, from the finiteness of the dissipation integral, Corollary \ref{dissint} and the above, we conclude $$\int_0^{\infty}|\lb u_t(\tau),\eta\rb|^2 d\tau < \infty.$$ We  apply the Barbalat Lemma to the function $\int^t|\langle u_t,\eta \rangle_{\Omega}|^2d\tau$ (as in \cite{Abhishek,ryz}), and conclude that  \begin{equation}\label{needed1} \lim_{t \to + \infty} k_0\lb u_t,\eta \rb_{\Omega}  = 0 \end{equation} for any $\eta \in C_0^{\infty}(\Omega)$. 

This distributional convergence $u_t$ to 0 in \eqref{needed1} and the strong convergence on subsequences $u_t(t_n) \rightarrow \hat w$ (from the existence of the compact attracting set)   imply that the strong limit $\hat w =0$. Since every sequence of times $t_n$ has a convergent subsequence such that $||u_t({t_{n_k}})||_0 \to 0$ as $k \to \infty$, we infer that  $\ds \lim_{t\to\infty} ||u_t(t)||_0^2 = 0,$ as desired for Point (3.) in Theorem \ref{convergenceprops}.}

 Now, as described above, from the existence of the attracting set for  the plate component we conclude strong convergence on a subsequence  of
\begin{equation}\label{one} \| u(t_{n_k}) -  \hat{u}  \|_{2,\Omega}  \rightarrow 0\end{equation} when $k \rightarrow \infty $. As for the lower order term, the following bound is clear
\begin{equation}\label{two}
\|u(t_n+\epsilon)-u(t_n)\|_{0,\Omega} \le \int_{t_n}^{t_n+\epsilon} \|u_t(\tau)\|_{0,\Omega}d\tau \le \epsilon\cdot \sup_{\tau \in [t_n,t_n+\epsilon]}\|u_t(\tau)\|_{0,\Omega}
\end{equation}
Thus \begin{align}
\sup_{\tau \in [-c,c]}\|u(t_n+\tau)-\hat u\|_{0,\Omega}\le& ~\sup_{\tau \in[-c,c]} \|u(t_n+\tau)-u(t_n)\|_{0,\Omega}+\|u(t_n)-\hat u\|_{0,\Omega}\\ \nonumber  & \to 0 ~\text{ for any fixed}~ c > 0
\end{align} along the subsequence $t_{n_k}$ by \eqref{one} and \eqref{two} above. More is true: by the interpolation inequality
$$\|u(t)\|_{2-\delta,\Omega} \le \|u(t)\|^{\delta/2}_{0,\Omega}\|u(t)\|^{1-\delta/2}_{2,\Omega},$$
and the boundedness of $\left\{\|u(t)\|_{2,\Omega}~:~t \ge 0 \right\}$ we see that \begin{align}
\sup_{\tau \in [-c,c]}\|u(t_n+\tau)-\hat u\|_{2-\delta,\Omega} \to 0 ~\text{ for any fixed}~ c > 0, \delta\in (0,2).
\end{align}
\end{proof}
\begin{remark} As seen in \cite{Abhishek}, the convergence \begin{align}
\sup_{\tau \in [-c,c]}\|u(t_n+\tau)-\hat u\|_{2-\delta,\Omega} \to 0 ~\text{ for any fixed}~ c > 0, ~\delta\in (0,2).
\end{align}
can be improved to $\delta=0$ by a contradiction argument, though this is not used here. \end{remark}

\subsection{Weak Convergence to Equilibria Set}

Now, to further characterize the flow limit point, we return to the formula in Theorem \ref{flowformula}. The key result is the identification of the limit as a {\it statinoary point. }
\begin{lemma}\label{characterize}
The weak limit point $(\hat \phi, \hat \psi) \in Y_{fl,\rho}$ in Theorem \ref{convergenceprops} above has $\hat \psi=0$ in the sense $L^2(K_{\rho})$ for any $\rho>0$.
\end{lemma}
\begin{proof}
The flow solution coming from the {\em initial data}, denoted $(\phi^*(t),\phi_{t}^*(t))$, tends to zero in the local flow energy sense, owing to Huygen's principle. Thus from Point (1.) of Theorem \ref{convergenceprops} we have that $\phi^{**}_t(t_{n_k}) \rightharpoonup \hat \psi$ (since $\phi=\phi^*+\phi^{**}$). To identify $\hat \psi$ with zero, we differentiate the flow formula \eqref{phidef} to obtain \eqref{thisonenow} (holding in distribution).
 For a fixed $\rho >0$, we multiply the function $\phi_{t}^{**}(\xb,t)$ by a smooth function $\zeta \in C_0^{\infty}(K_{\rho})$ and integrate by parts in space---in \eqref{thisonenow} move $\partial_x$ onto $\zeta$ in the third term and  $M_{\theta}$ onto $\zeta$ in term four.
This results in the bound
\begin{equation}\label{whatwe}
|(\phi^{**}_{t},\zeta)_{L^2(K_{\rho})}| \le C(\rho)\sup_{\tau \in [0,t^*]}\|u_{t}(t-\tau)\|_{0,\Omega}\|\zeta\|_{1,K_{\rho}}
\end{equation}

From this point, we utilize the previously established fact that $u_{t}(t) \to 0$ in $L^2(\Omega)$ (Theorem \ref{convergenceprops}), and hence $(\phi_{t}^{**}(t),\zeta)_{L^2(K_{\rho})} \to 0$, so $\phi_t^{**}(t) \to 0$ in $\mathcal D'(K_{\rho})$. This gives that $\phi_{t}(t_{n_k}) \rightharpoonup 0$ in $L^2(K_{\rho})$, and we identify the limit point $\widehat \psi = 0$ in $L^2(K_{\rho})$.
\end{proof}

We must now show that the obtained weak limit $S_{t_n}(y_0) \rightharpoonup (\hat \phi, 0; \hat u,0)$ in $Y_{\rho}$ provides a weak solution to the stationary problem. The proof proceeds as in \cite{Abhishek}, but we sketch it for completeness.
\begin{lemma}\label{staticsols}
The pair $(\hat \phi,\hat u)$ in Theorem \ref{convergenceprops} is a weak solution to the stationary problem \eqref{static}.
\end{lemma}
\begin{proof} (Sketch) Consider a strong solution to \eqref{flowplate}.
We begin by multiplying the system \eqref{flowplate} by $\eta \in C_0^{\infty}(\Omega)$ and $\mu \in C_0^{\infty}(\realsthree_+)$, and integrate over the respective domains. This yields
\begin{equation}\begin{cases}
\lb u_{tt},\eta\rb +\lb \Delta u, \Delta \eta\rb +\lb u_t,\eta\rb +\lb f(u),\eta\rb  = \lb p_0,\eta\rb -\lb \gamma_0[\phi_t+U\phi_x],\eta\rb  \\
\left(\phi_{tt},\mu \right)+U(\phi_{tx},\mu)+U(\phi_{xt},\mu)-U^2(\phi_{x},\mu_x)=-(\nabla \phi, \nabla \mu)+\lb u_t+Uu_x, \gamma_0[\mu]\rb .
\end{cases}
\end{equation}
Now, we consider the above relations evaluated on some points $t_n$ (identified as a subsequence for which the various convergences hold).  We then integrate in the time variable from $t_n$ to $t_n+c$. After temporal integration, the resulting identity can be justified on generalized solutions by the approximation definition of thereof. We then pass to limit in $n$, using Theorem \ref{convergenceprops}.

Limit passage on linear terms is clear, owing to the main convergence properties for the plate component in the Theorem \ref{convergenceprops}. The local Lipschitz property of $f_V$ allows us to pass with the limit on the nonlinear term (this is by now standard, \cite{springer}). We then arrive at the following relations:
\begin{equation}\begin{cases}
\lb \Delta \hat u, \Delta w\rb+\lb f_V(u),w\rb = \lb p_0,w\rb+U\lb \gamma_0[\hat \phi],w_x\rb \\
(\nabla \hat \phi, \nabla \mu)-U^2(\hat \phi_{x},\mu_x)=U\lb \hat u_x, \gamma_0[\mu]\rb.
\end{cases}
\end{equation}
This implies that our limit point $(\hat u,0; \hat \phi,0)$  for  $\left(\phi(t_n),\phi_t(t_n);u(t_n),u_t(t_n)\right)$ is a weak {\em stationary solution} to the flow-plate system \eqref{flowplate}. We have thus shown: any trajectory contains a sequence of times $t_n \to \infty$ such that the restricted trajectory  converges  {\it weakly} to a stationary solution.  \end{proof}
\begin{remark}
Knowing that the system is {\em gradient} would allow to reach this conclusion in a more straightforward way. However, this property is known only for the unperturbed problem $U=0$. 
\end{remark}

\noindent {\bf Summarizing Section \ref{techsec}}: Theorem \ref{convergenceprops} guarantees that for any $y_0 \in Y_{\rho_0}$ and for any sequence $t_n \to \infty$ and any $\rho>0$, there is weak convergence  in $Y_{\rho}$ along a subsequence $t_{n_k}$ of the trajectory $S_{t_{n_k}}(y_0)$ to some point $\hat y=( \hat \phi, \hat \psi; \hat u, \hat w) \in Y_{\rho}$. Moreover, point (3.) of Theorem \ref{convergenceprops} guarantees that $\hat w =0$. We then note that Lemma \ref{characterize} holds, and hence $\hat \psi=0$. With our weak limit point satisfying the weak form of the stationary problem in Lemma \ref{staticsols}, we conclude that we have obtained weak convergence to a weak solution of the stationary flow plate problem. We therefore have shown that any sequence $(u(t_n); \phi_t(t_n))$ with $t_n \to \infty$ contains a subsequence which converges to a stationary solution. A simple contradiction argument (see, e.g., \cite{ryz,springer,Abhishek})  gives that for any $\rho_0, \rho>0$ and any $y_0 \in Y_{\rho_0}$ we have $S_t(y_0) \rightharpoonup \mathcal N$ in the sense of $Y_{\rho}$. {\bf What remains to show to obtain the main result is the improvement of this convergence from weak to strong.}

\section{Microlocal Regularity of the hyperbolic  Neumann-Dirichlet map.}\label{micro}
The fundamental supporting result here---what finally allows us to overcome the issues encountered in previous stability analyses \cite{chuey,eliminating1,eliminating2}---is a sharp microlocal estimate for the Neumann lift. We note, for the first time, that the critical loss of Sobolev regularity in the so called characteristic  sector for the wave equation is compensated for by the plate dynamics. In this characteristic sector (defined below) we have $u_t ~\sim u_x \in H^1(\Omega)$. (Note: there is no loss of regularity for the Neumann wave equation in the hyperbolic and elliptic sectors.) 

		We will work with microlocally anisotropic spaces, as utilized heavily in the work of Tataru for trace estimates for the wave equation.
	Invoking the definition and notation from \cite{tataru}, we introduce the anisotropic space $X^s_\theta$. 
	Roughly, the notation $u\in X^s_\theta$ means that $u$ is $H^s$ in the characteristic sector (where time and space dual variables are comparable, given in \eqref{sectors} as $(c)$), but $u \in H^{s+\theta}$ in the other two sectors $(e)$ and $(h)$. 
	
	Let us introduce dual variables $\sigma \in \mathbb R,~ \eta\in \mathbb R^2 $ so that by Fourier-Laplace transformation
	\begin{align*} t \to &~ \tau=\xi + i\sigma,~~~
		(x,y) \to  i \boldsymbol \mu = i(\mu_1,\mu_2),
	\end{align*}
	with $\xi$ fixed and sufficiently large.
	In this section, we shall utilize the space $X^{1/2}_{-1/2}$ which consists of functions  $ h(x,y;t)$ with the properties: 
	\begin{itemize}
	\item ~$\ds e^{-\xi t} h\in L^2(\mathbb R^2 \times \mathbb R_+)$
	\item The Fourier-Laplace transform (defined explicitly below) $\hat{h}( \tau, \boldsymbol \mu)$    satisfies
	
	$$m( \sigma, \boldsymbol \mu )  \hat{h }(\tau, \boldsymbol \mu) \in L^2\big(\mathbb R_{+} ; H^{1/2} (\mathbb R^2)\big)$$  with  $m(\sigma,\boldsymbol \mu) = \hat {M}(t,x,y;t)$ for  ~$M(x,y;t)$  a zeroth order tangential  $\Psi$DO  on (the boundary) $\mathbb R^2$  with the property that $\ds \text{supp}\big\{m(\sigma, \boldsymbol \mu)\big\} \subset M_{C,c} \equiv  \big\{ c|\boldsymbol \mu| \leq |\sigma|\leq C |\boldsymbol \mu| \big\}$ 
	for some constants $c, C > 0 $. In addition, we assume that $m\equiv 1$ in $\ds M_{\frac{3C +c}{4}, \frac{4 c + C }{4} }$ \end{itemize}
	
	With that notation in place, we look at the following Neumann wave equation:
	\begin{equation}\label{flow-h}
		\begin{cases}
			\eta_{tt}=\Delta \eta & \text { in }~~ \realsthree_+ \times (0,T),\\
			\eta(0)=\eta_0;~~\eta_t(0)=\eta_1 & \text { in }~~ \realsthree_+\\
			\partial_z \eta = h^*(\xb,t)& \text{ on } ~~\{z=0\} \times (0,T),
		\end{cases} 
	\end{equation}
	{\bf Our goal in this section is to determine  the interior and the trace regularities of $\eta$ when the flow data $h^*$ has microlocally anisotropic regularity.} In particular, for $u_t$ as the plate component of a solution to \eqref{flowplate*}, we will have $u_t \in X^{1}_{-1}$. However, this provides an overcompensation for the Neumann loss (as discussed above); thus we only need consider Neumann data of the form
	 $h^*\in X^{1/2}_{-1/2}.$
We will show the following theorem: \begin{theorem}\label{mapprops} Any  weak solution to \eqref{flow-h}  satisfies the following a priori bound for $0 \le t \le T$:
	\begin{align}\label{trace*} \nonumber
		||\gamma_0 [\eta_t]||^2_{L^2(0,T; L^2(\R^2) } + ||\gamma_0 [\eta] ||^2_{L^2(0,T; H^{1}(\R^2))} +&	||\eta_t||^2_{L^2(0,T; L^2(\R^3_+) } + ||\eta ||^2_{L^2(0,T; H^{1}(\R^3_+))} \\
\le &	 C_T\big[  ||\eta(0)||^2_{H^1( \R^3_+)} + ||\eta_t(0)||^2+|| h^*||^2_{X^{1/2}_{-1/2}}\big].
\end{align}
\end{theorem}
\subsection{Proof of Theorem \ref{mapprops}}
To prove (\ref{trace*}), we critically exploit the  half-space structure of the problem in avoiding commutators and variable coefficients. 
	 We  apply superposition with respect to  initial and boundary data.
	\vskip.2cm
	\noindent{\bf Step 1}: Consider $h^* =0 $ and let $\eta$ be the corresponding response to the initial conditions.
	Applying \cite[Theorem 3]{miyatake}, we obtain
	\begin{multline}\label{ic}
		||\eta(t)||^2_{H^1(\R^3_+)} + ||\eta_t(t)||^2
		+ ||\gamma_0 [\eta_t]||^2_{L^2(0,T; H^{-1/2}(\R^2)) } + ||\gamma_0 [\eta] ||^2_{L^2(0,T; H^{1/2}(\R^2))} \\
		\leq C_T\left[ ||\eta(0)||^2_{H^1( \R^3_+)} + ||\eta_t(0)||^2\right].
	\end{multline}
	\noindent{\bf Step 2: } Consider  zero initial data and  arbitrary $ h^* \in X_{-1/2}^{1/2} $, and again let $\eta$ be the associated solution.
		Since  $\phi_0, \phi_1 =0 $ we then take the Fourier-Laplace transform with $\xi>0$ fixed:
	\begin{align*} t \to &~ \tau=\xi + i\sigma,~~~
		(x,y) \to  i \boldsymbol \mu = i(\mu_1,\mu_2),
	\end{align*}
denoting by $\widehat{\eta} = \widehat {\eta}(z, \boldsymbol \mu, \sigma)$ the Fourier-Laplace transform of $\eta$ in $x$, $y$ and $t$, i.e.
	$$
	\widehat {\eta}(z, \boldsymbol \mu, \tau) = \frac{1}{(2\pi )^2}\int_{{\R}^2} dx dy
	\int_{0} ^{+\infty} ~dt e^{-\tau t} \cdot e^{-i (x\mu_1+y\mu_2)}\cdot
	\eta (\mathbf x;t).
	$$
	This yields the equation \begin{equation*}
		\widehat{\eta}_{zz}=(|\boldsymbol \mu|^2+\xi^2-\sigma^2+2i\xi\sigma)\widehat{\eta},
	\end{equation*}
	with transformed boundary condition, ~
$\ds		\widehat{\eta}_z(z=0)=\widehat{h^*}(\boldsymbol \mu,\tau).$
	
	Solving the ODE in $z$ and choosing the decaying solution $z \to +\infty$,
	we have
	\begin{equation*} \widehat{\eta}(z, \boldsymbol \mu,\tau) =  - \frac{1}{\sqrt{s}}\widehat{h^*}( \boldsymbol \mu,\tau)\exp(-z\sqrt{s}),
	\end{equation*}
	with
	\begin{equation}\label{s}
		s \equiv | \boldsymbol   \mu|^2+\tau^2=\big( |   \boldsymbol \mu|^2-\sigma^2+\xi^2\big)+2i\xi \sigma;
	\end{equation}
	and $\sqrt{s}$ is chosen such that
	$ Re\sqrt{\tau^2 +| \boldsymbol \mu|^2}>0$ for $Re~ \tau = \xi >0$.
	On the boundary $z=0$ we have 
	\begin{equation}
		\widehat{\eta}(z=0, \boldsymbol \mu,\tau)=\frac{1}{\sqrt{s}}\widehat{h^*}( \boldsymbol \mu,\tau).
	\end{equation}
	Taking the time derivative amounts to premultiplying by $\tau=\xi + i \sigma$, and with a slight abuse of notation, we have 
	\begin{equation}
		\widehat{\eta}_t(z=0, \boldsymbol \mu,\tau)=\frac{\tau}{\sqrt{s}} \widehat{h^*}( \boldsymbol \mu,\tau),
		~~ \tau=\xi + i \sigma.
	\end{equation}
	Denoting the multiplier 
	\begin{equation}\label{mdef}
		\dfrac{\xi+i\sigma}{\sqrt{s}} \equiv m(\xi, \sigma, \boldsymbol \mu),
	\end{equation}
	we can infer the trace regularity of $\eta$ from $m(\xi,\sigma, \boldsymbol \mu)$:
	\begin{align}\label{needrts}
		|m(\xi,\sigma, \boldsymbol \mu)| = &~\Big| \dfrac{\xi+i\sigma}{\sqrt{s}}\Big|
		= \dfrac{\sqrt{\xi^2+\sigma^2}}{(|s|^2)^{1/4}} \\
		=&~\dfrac{\sqrt{\sigma^2+\xi^2}}{\big((| \boldsymbol \mu|^2-\sigma^2)^2+\xi^4+2\xi^2\sigma^2+2| \boldsymbol \mu|^2\xi^2\big)^{1/4}}\; .
	\end{align}
	WLOG, we take each of the arguments $\xi$, $\sigma$ and $| \boldsymbol \mu|$ to be positive and consider the partition of  the first quadrant of the $(| \boldsymbol \mu|,\sigma)$-plane into the following sectors: \begin{equation}\label{sectors}
		\begin{cases}(h)~~~ | \boldsymbol \mu| \le  \sigma/\sqrt{2}, & (hyperbolic)\\ (e)~~~ | \boldsymbol \mu| \ge \sqrt{2}\sigma, & (elliptic)\\ (c)~~~\sigma/\sqrt{2}< | \boldsymbol \mu| < \sqrt{2}\sigma & (characteristic)\end{cases}
	\end{equation}
	\subsubsection{Trace Regularity}
	Our goal in this section is to prove the trace component of \eqref{trace*}.	
	
	\begin{lemma}\label{roots}
		Let $s$ be as in \eqref{s}. Then, in sectors $(e)$ and $(h)$ we have the estimate
		\begin{equation}\label{sroot}
		\sqrt{|s|}\geq\frac{1}{2}|\xi|\left(1+\frac{| \boldsymbol \mu|^2}{\xi^2}\right)^{1/2}.
		\end{equation}
		and in sector $(c)$, we have the estimate
		\begin{equation}\label{sroot*}
			\sqrt{|s|}\geq|\xi|\left(1+\frac{| \boldsymbol \mu|^2}{\xi^2}\right)^{1/4}.
		\end{equation}
	\end{lemma}
	\begin{proof}[Proof of Lemma \ref{roots}]

		Also, from \eqref{needrts}, we may write
		\begin{equation*}
			\sqrt{|s|}=\left((| \boldsymbol \mu|^2-\sigma^2)^2+\xi^4+2\xi^2\sigma^2+2| \boldsymbol \mu|^2\xi^2\right)^{1/4}.
		\end{equation*}
		Hence, in sectors $(e)$ and $(h)$, we obtain
		\begin{equation}
			\begin{split}
				\sqrt{|s|}\geq\left((| \boldsymbol \mu|^2-\sigma^2)^2+\xi^4+2| \boldsymbol \mu|^2\xi^2\right)^{1/4}&\geq \left(\frac{| \boldsymbol \mu|^4}{4}+\xi^4+2| \boldsymbol \mu|^2\xi^2\right)^{1/4}\\
				&\geq \left(\frac{| \boldsymbol \mu|^2}{2}+\xi^2\right)^{1/2}\\
				&\geq \frac{1}{2}|\xi|\left(1+\frac{| \boldsymbol \mu|^2}{\xi^2}\right)^{1/2}
			\end{split}
		\end{equation}
		In sector $(c)$, we have
		\begin{equation}
			\begin{split}
				\sqrt{|s|}\geq\left(\xi^4+2\xi^2\sigma^2+2| \boldsymbol \mu|^2\xi^2\right)^{1/4}&\geq|\xi|^{1/2}\left(\xi^2+2\sigma^2+2| \boldsymbol \mu|^2\right)^{1/4}\\
				&\geq |\xi|^{1/2}\left(\xi^2+3| \boldsymbol \mu|^2\right)^{1/4}\\
				&\geq |\xi|\left(1+\frac{| \boldsymbol \mu|^2}{\xi^2}\right)^{1/4}.
			\end{split}
		\end{equation}
		This proves \eqref{sroot}.
	\end{proof}
	
	\begin{lemma}\label{le:m}
		Let $m(\xi,\sigma, \boldsymbol \mu)$ be defined as in \eqref{mdef}. Then, in sectors $(e)$ and $(h)$(as in \eqref{sectors}), we have the estimate
		\begin{equation}\label{m-est0}
		|m(\xi,\sigma, \boldsymbol \mu)|\le 2, ~~\text{ for}~~ \xi\neq 0,
		\end{equation}
		and in sector $(c)$
		\begin{equation}\label{m-est}
			|m(\xi,\sigma, \boldsymbol \mu)|\le 2 \left[ 1+\frac{| \boldsymbol \mu|^2}{\xi^2}\right]^{1/4}~~~
			\text{ for}~~ \xi\neq 0.
		\end{equation}
	\end{lemma}
	\begin{proof}[Proof of Lemma \ref{le:m}]
		We consider the partition of  the first quadrant of the $(| \boldsymbol \mu|,\sigma)$-plane as in \eqref{sectors}.
		In cases $(e)$ and $(h)$ above, we
		can write
		\begin{equation*} 
			|m(\xi,\sigma, \boldsymbol \mu)| \le \dfrac{\sqrt{\sigma^2+\xi^2}}{\Big(\big| | \boldsymbol \mu|^2-\sigma^2\big|^2+\xi^4+2\xi^2\sigma^2\Big)^{1/4}}  \le \dfrac{\sqrt{\sigma^2+\xi^2}}{\Big(\sigma^4/4+\xi^4+2\xi^2\sigma^2\Big)^{1/4}} \le 2.
		\end{equation*}
		In case $(c)$ we have
		\begin{align*}
			|m(\xi,\sigma, \boldsymbol \mu)| \le &\dfrac{\sqrt{\sigma^2+\xi^2}}{\big(\xi^4+2\xi^2\sigma^2+2| \boldsymbol \mu|^2\xi^2\big)^{1/4}}
			\le \dfrac1{|\xi|^{1/2}}
			\dfrac{\sqrt{\sigma^2+\xi^2}}{\big(\xi^2+2\sigma^2+2|\boldsymbol \mu|^2\big)^{1/4}}
			\\
			\le &\dfrac1{|\xi|^{1/2}}
			\dfrac{\sqrt{\sigma^2+\xi^2}}{\big(\xi^2+3\sigma^2\big)^{1/4}}\le
			\dfrac1{|\xi|^{1/2}}
			\big(\sigma^2+\xi^2\big)^{1/4}\le \dfrac1{|\xi|^{1/2}}
			\big(2|\boldsymbol\mu|^2+\xi^2\big)^{1/4}.
		\end{align*}
		This implies the  estimate in \eqref{m-est}.
	\end{proof}
	We now complete the proof of the trace component in the following lemma.
	\begin{lemma}\label{tracereg}
		Let $\eta$ satisfy \eqref{flow-h} with zero initial data and $h^*\in X^{1/2}_{-1/2}$. Then, we have the following estimate
		\begin{equation}\label{trreg}
			||\gamma_0 [\eta_t]||^2_{L^2(0,T; L^2(\R^2) } + ||\gamma_0 [\eta] ||^2_{L^2(0,T; H^{1}(\R^2))}
			\leq C_T\big[ || h^*||^2_{X_{-1/2}^{1/2}}  \big].
		\end{equation}
	\end{lemma}
	\begin{proof}[Proof of Lemma \ref{tracereg}]
		We bound each term on the LHS of \eqref{trreg} separately.\\
		By the inverse Fourier-Laplace transform we have that
		\[
		e^{-\xi t}\eta_t (x,y,z=0,t) = \frac{1}{2\pi }\int_{{\R}^2} d\boldsymbol \mu
		\int_{-\infty} ^{\infty} d\sigma
		e^{i\sigma t} \cdot e^{i( x\mu_1+y\mu^2)}\cdot
		m(\xi,\sigma,\boldsymbol \mu)
		\widehat {h^*} (\boldsymbol \mu, \xi +i\sigma)
		\]
		and
		\[
		e^{-\xi t}\eta (x,y,z=0,t) = \frac{1}{2\pi }\int_{{\R}^2} d\boldsymbol \mu
		\int_{-\infty} ^{\infty} d\sigma
		e^{i\sigma t} \cdot e^{i( x \mu_1+y\mu^2)}\cdot
		\frac{1}{\sqrt{s}}
		\widehat {h^*} (\boldsymbol\mu, \xi +i\sigma).
		\]
		Thus by the Parseval  equality we obtain
		\begin{equation}\label{eqeq}
		n(\xi;\eta_t)\equiv\frac{1}{2\pi}\int_0^{+\infty}\|e^{-\xi t}\eta_t (z=0,t)\|^2_{L^2(\R^2)}dt =\int_{{\R}^2} d\boldsymbol\mu
		\int_{-\infty} ^{\infty} d\sigma
		|m(\xi,\sigma,\boldsymbol\mu)|^2
		|\widehat {h^*} (\boldsymbol\mu, \xi +i\sigma)|^2
		\end{equation}
		and
		\begin{equation}\label{eqeq1}
			p(\xi;\eta)\equiv\frac{1}{2\pi}\int_0^{+\infty}\|e^{-\xi t}\eta (z=0,t)\|^2_{H^{1}(\R^2)}dt =\int_{{\R}^2} d  \boldsymbol \mu
			\int_{-\infty} ^{\infty} d\sigma
			\dfrac{(1+| \boldsymbol \mu|^2)}{|s|}
			|\widehat {h^*} ( \boldsymbol \mu, \xi +i\sigma)|^2.
		\end{equation}
		Proving \eqref{trreg} is equivalent to showing the RHS in \eqref{eqeq} and \eqref{eqeq1} are bounded.
		
		We first look at \eqref{eqeq}.
		From Lemma \ref{le:m}, in sector $(c)$, we have
		\begin{equation}\label{sectb}
		|m(\xi,\sigma, \boldsymbol \mu)|^2\leq 4(1+\xi^{-2}| \boldsymbol \mu|^2)^{1/2}\leq 4(1+\xi^{-2})^{1/2}(1+| \boldsymbol \mu|^{2})^{1/2}.
		\end{equation}
		Using \eqref{sectb} and \eqref{m-est0} and splitting the the RHS of \eqref{eqeq} based on sectors $(e)$, $(h)$ and $(c)$, we obtain
		\begin{equation}\label{tr1}
			\begin{split}
				\int_{{\R}^2} d \boldsymbol \mu\int_{-\infty} ^{\infty} d\sigma|m(\xi,\sigma, \boldsymbol \mu)|^2&|\widehat {h^*} ( \boldsymbol \mu, \xi +i\sigma)|^2=8\int_{{\R}^2} d \boldsymbol \mu\int_{0\leq\sigma\leq \frac{| \boldsymbol \mu|}{\sqrt{2}}} d\sigma^2|\widehat {h^*} ( \boldsymbol \mu, \xi +i\sigma)|^2\\
				&+8(1+\xi^{-2})^{1/2}\int_{{\R}^2} d \boldsymbol \mu\int_{\frac{| \boldsymbol \mu|}{\sqrt{2}}<\sigma< \sqrt{2}| \boldsymbol \mu|} d\sigma(1+| \boldsymbol \mu|^{2})^{1/2}|\widehat {h^*} ( \boldsymbol \mu, \xi +i\sigma)|^2\\
				&+8\int_{{\R}^2} d \boldsymbol \mu\int_{\sigma\geq \sqrt{2}| \boldsymbol \mu|} d\sigma^2|\widehat {h^*} ( \boldsymbol \mu, \xi +i\sigma)|^2.
			\end{split}
		\end{equation}
		Now, using the fact that $h^*\in X^{1/2}_{-1/2}$, we see that the RHS of \eqref{tr1} is bounded.
		
		\par
		We now look at \eqref{eqeq1}
		From Lemma \ref{roots}, in sectors $(e)$ and $(h)$, we have
		\begin{equation}\label{sectab}
		\dfrac{(1+| \boldsymbol \mu|^2)}{|s|}\leq\dfrac{4(1+| \boldsymbol \mu|^2)}{|\xi|^2(1+\xi^{-2}| \boldsymbol \mu|^2)}\leq\dfrac{4(1+| \boldsymbol \mu|^2)}{\xi^2(1+\xi^{-2})(1+| \boldsymbol \mu|^2)}\leq \dfrac{4}{1+\xi^2},
		\end{equation}
		and, in sector $(c)$, we have
		\begin{equation}\label{sectc}
		\dfrac{(1+| \boldsymbol \mu|^2)}{|s|}\leq\dfrac{(1+| \boldsymbol \mu|^2)}{|\xi|^2(1+\xi^{-2}| \boldsymbol \mu|^2)^{1/2}}\leq\dfrac{(1+| \boldsymbol \mu|^2)}{\xi^2(1+\xi^{-2})^{1/2}(1+| \boldsymbol \mu|^2)^{1/2}}\leq \dfrac{(1+| \boldsymbol \mu|^2)^{1/2}}{\xi(1+\xi^2)^{1/2}},
		\end{equation}
		Splitting the RHS of \eqref{eqeq1} based on sectors $(e)$, $(h)$ and $(c)$ and then applying \eqref{sectab} and \eqref{sectc}
		\begin{equation}\label{tr2}
			\begin{split}
				\int_{{\R}^2} d \boldsymbol \mu\int_{-\infty} ^{\infty} d\sigma\dfrac{(1+| \boldsymbol \mu|^2)}{|s|}&|\widehat {h^*} ( \boldsymbol \mu, \xi +i\sigma)|^2=\dfrac{4}{1+\xi^2}\int_{{\R}^2} d \boldsymbol \mu\int_{0\leq\sigma\leq \frac{| \boldsymbol \mu|}{\sqrt{2}}} d\sigma^2|\widehat {h^*} ( \boldsymbol \mu, \xi +i\sigma)|^2\\
				&+\dfrac{1}{\xi(1+\xi^2)^{1/2}}\int_{{\R}^2} d \boldsymbol \mu\int_{\frac{| \boldsymbol \mu|}{\sqrt{2}}<\sigma< \sqrt{2}| \boldsymbol \mu|} d\sigma(1+| \boldsymbol \mu|^{2})^{1/2}|\widehat {h^*} ( \boldsymbol \mu, \xi +i\sigma)|^2\\
				&+\dfrac{4}{1+\xi^2}\int_{{\R}^2} d \boldsymbol \mu\int_{\sigma\geq \sqrt{2}| \boldsymbol \mu|} d\sigma^2|\widehat {h^*} ( \boldsymbol \mu, \xi +i\sigma)|^2.
			\end{split}
		\end{equation}
		Now, using the fact that $h^*\in X^{1/2}_{-1/2}$, we observe that the RHS of \eqref{tr2} is bounded.
	\end{proof}

	\subsection{Interior Regularity}
	 We now prove the interior portion of the bound in \eqref{trace*} in Theorem \ref{mapprops}. Completing this will finish the proof of Theorem \ref{mapprops}.
		\begin{lemma}\label{inreg}
		Let $\eta$ satisfy \eqref{flow-h} with zero initial data and  and $h^*\in X^{1/2}_{-1/2}$. Then, we have the following estimate
		\begin{equation}\label{intreg}
			||\eta_t||^2_{L^2(0,T; L^2(\R^3_+) } + ||\eta ||^2_{L^2(0,T; H^{1}(\R^3_+))}
			\leq C_T || h^*||^2_{X^{1/2}_{-1/2}}.
		\end{equation}
	\end{lemma}
	\begin{proof}
		We bound each term on the LHS of \eqref{trreg} separately.
		By the inverse Fourier-Laplace transform we have that
		\[
		e^{-\xi t}\eta_t (x,y,z,t) = \frac{1}{2\pi }\int_{{\R}^2} d \boldsymbol \mu
		\int_{-\infty} ^{\infty} d\sigma
		e^{i\sigma t} \cdot e^{i( x  \mu_1+y  \mu_2)}\cdot
		m(\xi,\sigma, \boldsymbol \mu)\exp(-z\sqrt{s})
		\widehat {h^*} ( \boldsymbol \mu, \xi +i\sigma)
		\]
		and
		\[
		e^{-\xi t}\eta (x,y,z,t) = \frac{1}{2\pi }\int_{{\R}^2} d \boldsymbol \mu
		\int_{-\infty} ^{\infty} d\sigma
		e^{i\sigma t} \cdot e^{i( x  \mu_1+y  \mu_2)}\cdot
		\frac{\exp(-z\sqrt{s})}{\sqrt{s}}
		\widehat {h^*} ( \boldsymbol \mu, \xi +i\sigma).
		\]
		Thus by the Parseval  equality we obtain
		\begin{equation}\label{eq}
			\begin{split}
			n(\xi;\eta_t)&\equiv\frac{1}{2\pi}\int_0^{+\infty}\|e^{-\xi t}\eta_t (z,t)\|^2_{L^2(\R^3_+)}dt\\
			 &=\int_0^\infty dz\int_{{\R}^2} d \boldsymbol \mu
			\int_{-\infty} ^{\infty} d\sigma
			|m(\xi,\sigma, \boldsymbol \mu)|^2
			|\widehat {h^*} ( \boldsymbol \mu, \xi +i\sigma)|^2\exp(-2zRe(\sqrt{s}))
			\end{split}
		\end{equation}
		and
		\begin{equation}\label{eq1}
			\begin{split}
				p(\xi;\eta)&\equiv\frac{1}{2\pi}\int_0^{+\infty}\|e^{-\xi t}\eta (z=0,t)\|^2_{H^{1}(\R^3_+)}dt\\
				&=\int_0^\infty dz\int_{{\R}^2} d \boldsymbol \mu
				\int_{-\infty} ^{\infty} d\sigma
				\dfrac{(1+| \boldsymbol \mu|^2)}{|s|}
				|\widehat {h^*} ( \boldsymbol \mu, \xi +i\sigma)|^2\exp(-2zRe(\sqrt{s})).
			\end{split}
		\end{equation}
		Proving \eqref{intreg} is equivalent to showing the RHS in \eqref{eq} and \eqref{eq1} are bounded.\\
		Since $Re(\sqrt{s})>0$, from Lemma \ref{roots} we have
		\begin{equation}\label{exp}
			\sqrt{2}Re(\sqrt{s})=\sqrt{|s|+Re(s)}\geq\sqrt{|s|}\geq\frac{\xi}{2}.
		\end{equation}
		We first look at \eqref{eq}.
		Applying \eqref{exp} to \eqref{eq}, we obtain
		\begin{equation}\label{int1}
			\begin{split}
				\int_0^\infty dz\int_{{\R}^2}& d \boldsymbol \mu\int_{-\infty} ^{\infty} d\sigma|m(\xi,\sigma, \boldsymbol \mu)|^2|\widehat {h^*} ( \boldsymbol \mu, \xi +i\sigma)|^2\exp(-2zRe(\sqrt{s}))\\
				&\leq \int_0^\infty dz\exp\left(-\dfrac{z\xi}{\sqrt{2}}\right)\int_{{\R}^2} d \boldsymbol \mu\int_{-\infty} ^{\infty} d\sigma|m(\xi,\sigma, \boldsymbol \mu)|^2|\widehat {h^*} ( \boldsymbol \mu, \xi +i\sigma)|^2.
			\end{split}
		\end{equation}
		Now, using \eqref{tr1} and the fact that $h^*\in X^{1/2}_{-1/2}$, we see that the RHS of \eqref{int1} is bounded.
		
		\par
		Similarly, applying \eqref{exp} to \eqref{eq1}, we obtain
		\begin{equation}\label{int2}
			\begin{split}
				\int_0^\infty dz\int_{{\R}^2}& d \boldsymbol \mu\int_{-\infty} ^{\infty} d\sigma\dfrac{(1+| \boldsymbol \mu|^2)}{|s|}|\widehat {h^*} ( \boldsymbol \mu, \xi +i\sigma)|^2\exp(-2zRe(\sqrt{s}))\\
				&\leq \int_0^\infty dz\exp\left(-\dfrac{z\xi}{\sqrt{2}}\right)\int_{{\R}^2} d \boldsymbol \mu\int_{-\infty} ^{\infty} d\sigma\dfrac{(1+| \boldsymbol \mu|^2)}{|s|}|\widehat {h^*} ( \boldsymbol \mu, \xi +i\sigma)|^2.
			\end{split}
		\end{equation}
	Using \eqref{tr2} and the fact that $h^*\in X^{1/2}_{-1/2}$, we see that the RHS of \eqref{int2} is bounded appropriately.
	\end{proof}

\subsection{Change of Variables and Final Estimate}
We have shown the desired (interior and boundary) microlocal estimates for the wave equation in \eqref{flow-h}. But it is easy  to move back to the perturbed wave equation in \eqref{flowplate}$_(3)$--\eqref{flowplate}$_(5)$, of course with well-posedness in hand, with the flow component given by:
	\begin{equation}\label{flowplate*}\begin{cases}
			(\partial_t+U\partial_{x_3})^2\phi=\Delta \phi & \text { in }~~ \realsthree_+ \times (0,T),\\
			\phi(0)=\phi_0;~~\phi_t(0)=\phi_1 & \text { in }~~ \realsthree_+\\
			\partial_z \phi = h& \text{ on } ~~\{z=0\} \times (0,T).
		\end{cases}
	\end{equation} 
One can see that the function $\eta(\xb,t)=\phi(\xb +Ut e_1, t)\equiv \phi(x +Ut, y,z, t)$
possesses the same regularity properties and solves the problem
\begin{equation}\label{flow-h*}
\begin{cases}
\partial_t^2\eta=\Delta \eta & \text { in } \realsthree_+,\\
\eta(0)=\phi_0;~~\eta_t(0)=\phi_1+U\partial_x\phi_0, \\
\Dn \eta = h^{*}& \text{ on } \realstwo,
\end{cases}
\end{equation}
where  $ h^{*}(\xb,t)=h(\xb +Ut e_1, t)$ also has the same regularity properties as $h$, for instance in $L^2(0,T;  L^2(\R^2))$, $\forall\, T>0$.
Thus we obtained the desired result for the perturbed wave equation. 
\begin{theorem}\label{mapprops*} Any weak solution to \eqref{floweq*} with zero initial data  satisfies the following a priori bound:
	\begin{equation}\label{trace}\small
		||\gamma_0 [ \phi_t]||^2_{L^2(0,T; L^2(\R^2) } + ||\gamma_0 [ \phi] ||^2_{L^2(0,T; H^{1}(\R^2))} +	|| \phi_t||^2_{L^2(0,T; L^2(\R^3_+) } + || \phi ||^2_{L^2(0,T; H^{1}(\R^3_+))}
		\leq C_T || h||^2_{X^{1/2}_{-1/2}}.
\end{equation}
\end{theorem}
\noindent This theorem will be applied to the flow-plate solution taking $h=[u_t+Uu_x]_{\text{ext}}$ and waiting a sufficiently long time (with respect to the support of the initial data characterized by $\rho>0$). 
\begin{remark} We note that, from the point of view of stability (long-time behavior), the above bound may not be directly useful, owing to the constant $C_T$ appearing on the RHS.  We will only need to apply this estimate on time-translated intervals of a uniform size, i.e., $[t_n-a,t_n+a]$ for some fixed $a>0$. This is a critical point in the analysis here. \end{remark}

\section{Improving Convergence: Weak to Strong}\label{weaktostrong}
This is the final  portion of the proof of the main result, Theorem \ref{regresult}. With the technical preliminaries established (Section \ref{techsec}), and the improved microlocal mapping properties of the Neumann-Dirichlet  mapping (Section \ref{micro}), we now boost weak convergence in Lemma \ref{staticsols} to a strong convergence. We  outline the strategy here:
\begin{itemize}
\item We consider a given sequence of times $t_n \to \infty$, for which we have established a subsequential weak limit $(\hat\phi, 0; \hat u, 0)$ that is a weak solution to the stationary problem. We consider the wave component of the PDE system on its own, where we take the difference variables $\psi = \phi-\hat \phi$ and $v=u-\hat u$. Since the wave equation itself is linear, we may apply our microlocal regularity results for the resulting $(\psi,v)$ system. 
 The goal is to show that ~$(\psi(t_n),v(t_n)) \to (0,0)$ strongly in $Y_{\rho}$ (perhaps on a further subsequence).
\item We consider the Neumann mapping bound on time intervals of the form $[-c+t_n,t_n+c]$, where $c>0$ is arbitrary; we then apply the microlocal result. Since these ``sliding'' intervals are all of the same length, we can invoke the semigroup property for tight control of constants that depend on $T$, which is to say, only the length of this interval matters.
\item We will then obtain that $\psi$ goes to zero (via convergence of the plate dynamics) in the sense of $L^2\big([-c+t_n,t_n+c];Y_{\rho}\big)$. We must convert this information to point-wise convergence along the sequence $t_{n_k}$ for the flow. We finally show a ``converging together'' lemma that gives $E(t_n)$,~ $E(t_n+c)$ converge to the same value with $c$ arbitrary; this allows us to deduce the required point-wise information. 
\end{itemize}

\subsection{Convergence Through the Microlocal Regularity}\label{convmicro} Let $\{t_{n_k}\}$ be as in Lemma \ref{staticsols}, with $(u,\phi)$ satisfying \eqref{flowplate} and $(\hat{u},\hat{\phi})$ the associated solution the stationary problem to which it convergence along a given subsequence of times $t_{n_k}$. By linearity of the flow equation in \eqref{floweq*}, the variable $\psi=\phi-\hat{\phi}$ satisfies the same flow equation with the boundary data given by $$h=[v_t+Uv_x]_{\text{ext}}=[u_t+U(u-\hat u)_x]_{\text{ext}}.$$
		 Our ultimate goal is to show for any $\rho>0$
			\begin{equation}\label{req}
		||\phi(t_{n_k})-\hat{\phi}||_{W_1( K_{\rho} )}^2+||\phi_t(t_{n_k})||_{L^2( K_{\rho} )}^2\to 0,
		\end{equation}		
		We arrive at this through several steps, as outlined in the previous section.
		
		 We define, for a fixed $a>0$, time translates of the solutions along $t_n$. Consider the sequences $\{f_n\}$ and $\{g_n\}$ as:\\
		\begin{tabularx}{\textwidth}{XX}
			{\begin{align*}
			f_k(t)=&\begin{cases}
					v_{\text{ext}}(t_{n_k}+t),& ~t\in[-a,a]\\
					0,& ~\text{otherwise}
				\end{cases}
			\end{align*}} 
 			&
			{\begin{align*}
				g_k(t)=&\begin{cases}
				[v_t]_{\text{ext}}(t_{n_k}+t),&~ t\in[-a,a]\\
				0,& ~\text{otherwise}
			\end{cases}
			\end{align*}} 
		\end{tabularx}
and sequences $\{p_n\}$ and $\{q_n\}$ as\\
\noindent	
	\begin{tabularx}{\textwidth}{XX}
		{\begin{align*}
				p_k(t)=&\begin{cases}
					\psi(t_{n_k}+t),& t\in[-a,a]\\
					0,&~ \text{otherwise}
				\end{cases}
		\end{align*}} 
		&
		{\begin{align*}
				q_k(t)=&\begin{cases}
					\psi_t(t_{n_k}+t),& ~ t\in[-a,a]\\
					0,&~\text{otherwise}.
				\end{cases}
		\end{align*}} 
	\end{tabularx}

		Now, we wait a sufficient time so as to consider the flow equation with zero initial flow data on $K_{\rho}$ (Huygen's Principle). Applying the result on the Neumann lift (with zero flow data) in Theorem \ref{mapprops*}, for any fixed $a>0$, we obtain:
		\begin{equation}\label{mres}
			\begin{split}
				||\psi_t||^2_{L^2(t_{n_k}-a,t_{n_k}+a; L^2( K_{\rho} )) } + ||\nabla \psi ||^2_{L^2(t_{n_k}-a,t_{n_k}+a; L^2( K_{\rho} ))}&=||q_k||^2_{L^2(\mathbb{R}_+; L^2( K_{\rho} ) } + ||\nabla p_k||^2_{L^2(\mathbb{R}_+; L^2( K_{\rho} ))}\\
				&\leq C_T|| g_k +U{\left(f _k\right)}_{x}||^2_{X^{1/2}_{-1/2}}\\
				&\leq C_T\left[ || g_k ||^2_{X^{1/2}_{-1/2}}+U\|{\left(f _k\right)}_{x}||^2_{X^{1/2}_{-1/2}}  \right]
			\end{split}
		\end{equation}
		where above $T=2a$ (independent of $k$).
		
	We now look at the terms on the RHS. For simplicity, we use the notation $\chi(k)\defeq\chi_{[t_{n_k}-a,t_{n_k}+a]}$ for the characteristic function on the set $[t_{n_k}-a,t_{n_k}+a]$. Using the definition of the space $X^{1/2}_{-1/2}$ and Parseval's equality, we obtain
	\begin{equation}\label{mres1}
		\begin{split}
			|| g_k||^2_{X^{1/2}_{-1/2}}&=\int_{{\R}^2} d\mu\int_{0\leq|\sigma|\leq \frac{|\mu|}{\sqrt{2}}} d\sigma|\widehat{g_k}(\mu,\sigma)|^2+\int_{{\R}^2} d\mu\int_{\frac{|\mu|}{\sqrt{2}}<|\sigma|< \sqrt{2}|\mu|} d\sigma(1+|\mu|^{2})^{1/2}|\widehat{g_k}(\mu,\sigma)|^2\\
			&+\int_{{\R}^2} d\mu\int_{|\sigma|\geq \sqrt{2}|\mu|} d\sigma|\widehat{g_k}(\mu,\sigma)|^2\\
			&\leq 2\int_{{\R}^2} d\mu\int_{\mathbb{R}} d\sigma|\widehat{g_k}(\mu,\sigma)|^2+\int_{{\R}^2} d\mu\int_{\frac{|\mu|}{\sqrt{2}}<|\sigma|< \sqrt{2}|\mu|} d\sigma(1+|\mu|^{2})^{1/2}|\widehat{\chi(k)[v_t]_{\text{ext}} }(\mu,\sigma)|^2\\
			&\leq 2||g_k||^2_{L^2(\mathbb{R}_+; L^2(\mathbb R^2 )) }+\xi\int_{{\R}^2} d\mu\int_{\frac{|\mu|}{\sqrt{2}}<|\sigma|< \sqrt{2}|\mu|} d\sigma\sqrt{2}|\mu|(1+|\mu|^{2})^{1/2}|\widehat{\chi(k)v_{\text{ext}} }(\mu,\sigma)|^2\\
			&\leq 2||[u_t]_{\text{ext}} ||^2_{L^2(t_{n_k}-a,t_{n_k}+a; L^2( \mathbb R^2) )}+\xi\int_{{\R}^2} d\mu\int_{\mathbb{R}} d\sigma\sqrt{2}|\mu|(1+|\mu|^{2})^{1/2}|\widehat{\chi(k)v_{\text{ext}} }(\mu,\sigma)|^2\\
			&\leq 2||[u_t]_{\text{ext}} ||^2_{L^2(t_{n_k}-a,t_{n_k}+a; L^2( \mathbb R^2 )) }+\xi\|v_{\text{ext}} \|_{L^2(t_{n_k}-a,t_{n_k}+a; H^{3/2}( L^2(\mathbb R^2) )}
		\end{split}
	\end{equation}
From the convergence properties of $(u(t_{n_k}),u_t(t_{n_k}))$ of in Theorem \ref{convergenceprops} and the definition of $v$ variable , we see that the RHS of \eqref{mres1} approaches zero as $k$ approaches infinity. Also, from Theorem \ref{mapprops*}, and the fact that $H^{3/2}(\Omega) \hookrightarrow X^{1/2}_{-1/2}$, we see that
\begin{equation}\label{mres2}
	\|{\left(f _k\right)}_{x}||^2_{X^{1/2}_{-1/2}}\leq C_T\|v_{\text{ext}} ||^2_{L^2(t_{n_k}-a,t_{n_k}+a; H^{3/2}(\R^2) }\to 0,~ k\to \infty.
\end{equation}
Hence,
\begin{equation}\label{psilim}
	||\psi_t||^2_{L^2(t_{n_k}-a,t_{n_k}+a; L^2(K_{\rho}) } + ||\nabla \psi ||^2_{L^2(t_{n_k}-a,t_{n_k}+a; L^2(K_{\rho}))}\to 0, ~ k\to \infty.
\end{equation}
This establishes the desired convergence along time translated intervals of size $2a$ along $t_{n_k}$.

\subsection{Converging Together}\label{together} We must convert our $L^2$ type convergence on intervals to pointwise convergence along $t_n$. Note that this is precisely the type of information carried in the energy identity. We now look only at the energy for the flow equation evaluated on the difference trajectory $(\psi=\phi-\tilde\phi,v=u-\tilde u)$.
 We also recall that for generalized solutions, the energy $E^{sub}_{fl} \in C([0,T];\mathbb R_{\ge 0})$ along trajectories.
\begin{equation}\label{eid}
	E_{fl}^{sub}(t_2)=E_{fl}^{sub}(t_1)+\int_{t_1}^{t_2}\langle v_t(s)+Uv_x(s),\psi_t(s)\rangle_{\Omega} ds,
\end{equation}
where we recall ~$E_{fl}^{sub}(t)=\|\psi_{t}(t)\|^2+\|\nabla\psi(t)\|^2-U^2\|\psi_x(t)\|^2$~ is evaluated on the difference $\psi=\phi-\tilde \phi$. Using the energy identity \eqref{eid} on ~$[t_{n_k}-a,t_{n_k}+a]$, Parseval's, and Theorem \eqref{tracereg}, we see that
\begin{equation}\label{eninf}
	\begin{split}
		0&\leq \sup_{\tau\in[-a,a]}\left|E_{fl}^{sub}(t_{n_k}+\tau)-E_{fl}^{sub}(t_{n_k})\right|\\	
		&\leq\int_{t_{n_k}-a}^{t_{n_k}+a}|\langle v_t(s)+Uv_x(s),\psi_t(s)\rangle_{\Omega}| ds\\
		&\leq||\gamma_0 [\psi_t]||_{L^2(t_{n_k}-a,t_{n_k}+a; L^2(\R^2)) }||v_t+Uv_x||_{L^2(t_{n_k}-a,t_{n_k}+a; L^2(\R^2)) }\\
		&\leq \|g_k+Uf_{k,x}\|_{X_{-1/2}^{1/2}}^2\\
		&\leq \left(\|g_k\|_{X_{-1/2}^{1/2}}+U\|f_{k,x}\|_{X_{-1/2}^{1/2}}\right)^2\to 0,
	\end{split}
\end{equation}
where in the last line, we invoke \eqref{mres1} and \eqref{mres2}.
From Theorem \ref{convergenceprops}, we know that $S_{t_{n_k}}(y_0)$ converges weakly in $Y_{\rho}$, thus the sequence $E_{fl}^{sub}(t_{n_k})$ is  bounded. Possibly passing to a further subsequence and relabelling, we have a subsequential limit point (of real numbers)
\begin{equation}\label{elim}
	E_{fl}^{sub}(\tau+t_{n_k})\to L,\quad \text{ for some } L\geq 0,
\end{equation}
where the above convergence and limit point are uniform for $\tau \in [-a,a]$.

We now show that $L=0$ via contradiction, to obtain our final desired conclusion, \eqref{req}. Suppose, to the contrary, that $L\neq 0$. From \eqref{eninf} and \eqref{elim}, we may write $\displaystyle\inf_{\tau\in[-a,a]}E_{fl}^{sub}(t_{n_k}+\tau)\to L$ and hence for all $k$ large enough
\begin{equation}
	0<L/2<\inf_{\tau\in[-a,a]}E_{fl}(t_{n_k}+\tau).
\end{equation}
But this contradicts \eqref{psilim} since, from the definition of $E_{fl}^{sub}$, we have
\begin{equation*}
	2a\left(\inf_{\tau\in[-a,a]}E_{fl}^{sub}(t_{n_k}+\tau)\right)\leq ||\psi_t||^2_{L^2(t_{n_k}-a,t_{n_k}+a; L^2(\R^3_+) } + ||\psi ||^2_{L^2(t_{n_k}-a,t_{n_k}+a; W_{1}(K_{\rho})}\to 0, ~ k\to \infty.
\end{equation*}
Hence $L$ must be zero, and we obtain \eqref{req}.\\

Thus, we have improved the convergence of $S_{t_{n_k}}(y_0) \to (\hat \phi, 0; \hat u, 0) \in Y_{\rho}$, i.e., the limit in Part 1 of Theorem \ref{convergenceprops} is strong. By virtue of Theorem \ref{convergenceprops} and Lemma \ref{staticsols}, we have shown that, that given any sequence of times $\{t_n\}$ going to infinity, there exists a subsequence $\{t_{n_k}\}$ and an associated point $(\tilde{u},\tilde{\phi})\in\mathcal{N}$ for which
\begin{equation*}
	||\Delta[u(t_{n_k})-\tilde{u}]||_{2}^2+ ||u_t(t_{n_k})||_{0}^2+||\nabla [\phi(t_{n_k})- \tilde{\phi}]||_{L^2( K_{\rho} )}^2+||\phi_t(t_{n_k})||_{L^2( K_{\rho}) }\to 0.
\end{equation*}
We have thus obtained the statement of the main result, Theorem \ref{regresult}.

\section{Acknowledgements}  The authors wish to acknowledge and express appreciation to the National Science Foundation. The second author was partially supported by NSF-DMS 1713506. The third author was partially supported by NSF-DMS 1907620.

\section{Appendix}
The goal of this brief appendix is to discuss the features of our model that are essential for the main result, and also to demonstrate some ways in which our main result is sharp. 

\subsection{Essential Features of \eqref{flowplate}---Discussion}
First, we note that the presence of the flow $U \neq 0$ introduces the interactive energy $E_{int}$, which, although conserved, may not be positive; hence it must be controlled.  The presence of the flow (and also the functions $p_0$ and $F_0$) are the main features which give rise to the plate attractor, and the possibility of non-trivial stationary states. The von Karman nonlinearity is ``good'' precisely because it is ``strong enough'' to ensure boundedness of trajectories for all $U \in \mathbb R$. Thus, the removal of the physical plate nonlinearity would remove hope of producing the main result, unless strict smallness conditions were placed on $|U|$ as well as the forcing functions $p_0$ and $F_0$ (in appropriate norms).  Indeed, the requirement of $|U|<1$ is essential for our purposes here, and of course in line with the engineering motivation for this paper. Although the flow contributes $E_{int}$ as mentioned above, we do have a conservative system. With $|U|>1$, we do not have a good energy balance/identity, and energy-building temporal integrals appear on the RHS \cite{supersonic}. 

In the discussion below, we informally address two simplified problems that provide some insight into our result here, and, in some sense, demonstrate that our result is sharp with respect to the mode of convergence. 
We will discuss two simplified scenarios where $U=0$ and the plate is taken to be linear (regime of small, but perhaps rapid, deflections). In these cases, our mathematical flow-plate system becomes a true {\em acoustic problem} \cite{springer,graber}. In other words, we consider a true wave equation on the 3D space, with a damped linear plate embedded in a portion of its boundary. 

We first consider the half-space problem, as presented above in \eqref{flowplate} but taking $U=0$. We interpret our results from the previous sections to obtain a result. We then ask the questions of uniformity and rate of decay in this context. In the second portion of the Appendix, we restrict our attention to the bounded {\em acoustic chamber} \cite{springer} (and references therein). In both cases, we do not obtain uniformity of decay, for reasons explained below.

\subsection{$U=0$ Half-Space Problem}
Consider the linear acoustic problem on the 3D  half-space. This model is obtained by simply letting $U=0$ and eliminating the plate nonlinearity from \eqref{flowplate}. 
\begin{equation}\label{flowplatecons}\begin{cases}
		u_{tt}+\Delta^2u+k_0u_t = r_{\Omega}\gamma_0\big[\partial_t\phi \big]& \text { in }~~ \Omega\times (0,T),\\
		u(0)=u_0,~~u_t(0)=u_1 & \text{ in }~~\Omega,\\
		u=\Dn u = 0 & \text{ on } ~~\partial\Omega\times (0,T),\\
	\phi_{tt}=\Delta \phi & \text { in }~~ \realsthree_+ \times (0,T),\\
		\phi(0)=\phi_0,~~\phi_t(0)=\phi_1 & \text { in }~~ \realsthree_+\\
		\partial_z \phi = [\partial_tu]_{\text{ext}}& \text{ on } ~~\Omega \times (0,T).
	\end{cases}
\end{equation} 
One can view this model as a  half-space wave equation, with damping {\em active} on a small portion of the boundary through the plate equation. This model has a good energy relation, and trajectories are necessarily bounded, thus there is a weak limit point as $t\to \infty$. The unique stationary solution in this case is trivial (up to a constant, with respect to the Neumann flow problem).

 We see that the presence of the unbounded half-space is essential to invoke Huygen's principle, in order to dispense with the effects of the flow initial conditions (through waiting). The half-space is also needed for the closed-form reduction result (Theorem \ref{rewrite}) resulting in the plate dynamics with a memory potential on the RHS. 
  It will follow immediately that these plate dynamics converge uniformly (and exponentially) to zero---since the plate attractor exists in this case and is necessarily trivial.
 The decay of the plate can be lifted directly to the flow through the microlocal estimates, as used above; however, this can only be done on subsequences of time and the microlocal estimates invoked only on bounded (fattened) time intervals of uniform size. Thus there is no hope for a ``direct'' use of the microlocal estimate (as mentioned before, owing to the constant $C_T$) to lift the exponential decay of the plate to the flow on the infinite time horizon. Thus, there does not seem to any obvious uniformity of decay for solutions to the full system without invoking some sort of damping (e.g., viscosity) in the flow equation.

 A natural question would be to ask, what happens if $k_0=0$ above? Is it possible, after sufficiently long time has passed, that the transfer of energy from the plate to the flow would eventually ``dissipate the entire system energy at infinity?'' The answer to the very interesting question is in fact negative, as shown (by construction) in the very nice note \cite{igorexample}. It shows that flutter persists unless there is some damping on the structure. In that sense our result is optimal, since it  shows that flutter can be eliminated with an arbitrary small damping.

\subsection{Acoustic Chamber}
Let us now take $\Omega$ to be a flat portion of a bounded domain $\mathcal O$ which has boundary $\partial \mathcal O = \overline{\Omega}\cup \Omega^c$ that could be taken smooth, or even Lipschitz. Let $\mathbf n$ denote the unit outward normal to $\mathcal O$, so that $\mathbf n = -\mathbf e_3$ on the surface of $\Omega$.
\begin{equation}\label{flowplateconsbounded}\begin{cases}
		u_{tt}+\Delta^2u+k_0u_t = r_{\Omega}\gamma_0\big[\partial_t\phi \big]& \text { in }~~ \Omega\times (0,T),\\
		u(0)=u_0,~~u_t(0)=u_1 & \text{ in }~~\Omega,\\
		u=\Dn u = 0 & \text{ on } ~~\partial\Omega\times (0,T),\\
	\phi_{tt}=\Delta \phi & \text { in }~~ \mathcal O \times (0,T),\\
		\phi(0)=\phi_0,~~\phi_t(0)=\phi_1 & \text { in }~~ \mathcal O\\
		\partial_z \phi = u_t& \text{ on } ~~\Omega \times (0,T)\\
		\partial_{\mathbf n} \phi =0 & \text{ on } ~~\partial \mathcal O \setminus{\overline{\Omega}}.
	\end{cases}
\end{equation} 
This is a wave equation on a bounded domain, with damping on a portion of its bounded boundary {\em through the plate equation}. As in the previous subsection, this model has a good energy relation, a non-negative energy, and trajectories are necessarily bounded. We can say more.

First, it is immediate that the stationary solution is unique up to a flow constant $\phi = const.$ (as a pure Neumann problem).  Secondly, the dynamics in \eqref{flowplateconsbounded} can be seen to be {\em gradient}, i.e., with a strict Lyapunov function. If one considers the damper zero, in the sense $u_t \equiv 0$, then by the clamped plate boundary conditions we obtain $u \equiv 0$. In this case, we can rewrite the plate equation as:
 $$0 = r_{\Omega}\gamma_0\big[\partial_t\phi \big] \text { in }~~ \Omega\times (0,T).$$
 Then, time differentiating the resulting flow system (with $u_t=0$), we obtain:
 \begin{equation}\label{flowplateconsbounded*}\begin{cases}
	[\phi_t]_{tt}=\Delta [\phi_t] & \text { in }~~ \mathcal O \times (0,T),\\
		\partial_z [\phi_t] = 0 & \text{ on } ~~\Omega \times (0,T)\\
		\partial_{\mathbf n} \phi =0 & \text{ on } ~~\partial \mathcal O \setminus{\overline{\Omega}}.
	\end{cases}\end{equation}
	The above, taken with the condition that $\gamma_0[\phi_t]=0$ on $\Omega$, yields an overdetermined problem in $\phi_t$, which, by the standard hyperbolic theory provides that $\phi_t \equiv 0$. From this we deduce that $\phi_t \equiv 0$, and thus $\phi = const.$ is stationary. 

On the other hand, deducing {\em convergence} to a stationary state is not at all clear. Boundedness of trajectories (in the ``good'' energy) provides weak subsequential limit points in time. Additionally, we must have $u_t \rightharpoonup 0$ in $L^2(\Omega)$, since, if there were a subsequence of times where this did not hold, it would be seen to violate the finiteness of the dissipation integral (Corollary \ref{dissint}). Considering smooth data and using Ball's method of approximation for this linear problem allows one to boost the convergences. We obtain strong subsequential (in time) convergence to a limit point for a given subsequence $t_{n_k}$, and $u_t(t) \to 0 \in L^2(\Omega),~t\to \infty$ strongly. However, pushing this convergence to the flow to obtain a result like the main result in this treatment is not clear. 

The reduction formulae utilized here (as in \eqref{phidef} and \eqref{potential*}) to obtain our main result depend critically on Huygen's principle, waiting a sufficiently long amount of time on the  half-space until the effect of the initial condition is sent away. This of course is not available on the bounded domain. Concerning the effect of the boundary data itself, it is possible to prove microlocal estimates for this system, picking up some commutators and non-constant coefficients from the geometry. In particular, we will not have direct access to the ``lift'' given in \eqref{thisonenow} and used at several points herein. This particularly complicates the issue in passing from information about $u_t$ to that of $\phi_t$,  which convergence to a stationary state is predicated upon.

Although the boundedness of the domain is helpful from the point of compactness for the flow, the {\em loss of Hugyen's principle}---finite dependence on the signal---means that direct stablization of the plate (after waiting some time) is markedly more complicated. Said differently, the effects of the initial condition are not ``lost to dispersion", and we cannot invoke certain explicit formulae for the relevant N-to-D mapping. Any attempt to stabilize the initial conditions through the plate damper would seem to require a geometric condition on plate domain $\Omega$.

\end{document}